\documentclass{IET}
   \graphicspath{{images/}{../jpeg/}}
\usepackage{pgfplots}
\usepackage{algpseudocode}
\numberwithin{equation}{section}
\title{Robust Burg Estimation of Radar Scatter Matrix for Autoregressive structured SIRV based on Fr\'echet medians}

\makeatletter
\def\input@path{{images/}{./}}
\makeatother

\def\eqlaw{\buildrel d \over =}

\def \bSigma {\mathbf{\Sigma}}
\def \as {\overset{a.s.}{\rightarrow}}

\def \p {\mathbf{p}}
\def \w {\mathbf{w}}
\def \x {\mathbf{x}}
\def \y {\mathbf{y}}
\def \z {\mathbf{z}}

\newtheorem{proposition}{Proposition}
\newtheorem{remark}{Remark}

\author[1,*]{Alexis~Decurninge}
\affil{Thales Air Systems, Voie Pierre Gilles de Gennes 91470 Limours, France}

\author[1]{Fr\'{e}d\'{e}ric~Barbaresco}

\abstract{We  address the estimation of the scatter matrix of a scale mixture of Gaussian stationary autoregressive vectors. This is equivalent to consider the estimation of a structured scatter matrix of a Spherically Invariant Random Vector (SIRV) whose structure comes from an autoregressive modelization. The Toeplitz structure representative of stationary models is a particular case for the class of structures we consider.\\
For Gaussian autoregressive processes, Burg method is often used in case of stationarity for its efficiency when few samples are available. Unfortunately, if we directly apply these methods to estimate the common scatter matrix of $N$ vectors coming from a non-Gaussian distribution, their efficiency will strongly decrease. We propose then to adapt these methods to scale mixtures of autoregressive vectors by changing the energy functional minimized in the Burg algorithm.\\
Moreover, we study several approaches of robust modification of the introduced Burg algorithms, based on Fr\'echet medians defined for the Euclidean or the Poincar\'e metric, in presence of outliers or contaminating distributions. The considered structured modelization is motivated by radar applications, the performances of our methods will then be compared to the very popular Fixed Point estimator and OS-CFAR detector through radar simulated scenarios.}

\begin{document}

\maketitle


%

\section{Motivations}

\subsection{Context}
Real radar measurements of strong low grazing angle clutters such as ground or sea clutters showed that these clutters should be described by non-Gaussian distributions, especially heavy-tailed \cite{trizna91,watts85,billingsley93,watts16,watts16b}. The family of complex spherically invariant random vectors (SIRV), a subfamily of the elliptically symmetric distributions \cite{ollila12} (which contains a lot of classical distributions such as multivariate Gaussian, multivariate Cauchy distributions and multivariate K-distributions) is a useful generalization of Gaussian random vectors, inheriting of similar shape and location parameters. This family has been often used to modelize such radar clutters (see e.g. \cite{conte04,farina87,gini97,ollila12}). In this paper, we propose estimators of the scatter matrix of a zero-mean SIRV with a particular structure coming from an autoregressive (AR) modelization of the correlation between coordinates of the vector. This structured model is natural to describe the temporal correlation between radar pulse responses when signal is stationary and was already considered for example in \cite{bouvier95,michels95,trench64}.\\
In the following, the conjugate transpose operator applied on vectors or matrices is denoted by $(.)^*$ while the conjugate operator applied on a scalar $z$ is denoted by $\overline{z}$, vectors and matrices are denoted by \textbf{bold} letters, scalars by non-bold letters and $\as$ denotes the almost sure convergence.\\

Denote $\x=(x_1,...,x_d)^T\in\mathbb{C}^d$ a zero-mean SIRV representing for example the response of $d$ pulses for a radar spatial cell. $\x$ is then characterized by the existence of a Gaussian vector $\y$ of covariance $\bSigma$ and a positive non-Gaussian amplitude $\tau$ such that $\x\eqlaw \tau \y$; $\bSigma$ is called the scatter matrix \cite{yao73}. Moreover, we suppose that the scatter matrix has the same structure as the covariance matrix of a stationary AR vector of order $M$ (i.e. $\y$ is the trace of a Gaussian AR process of order $M$; see Section \ref{section_mixtures}).\\
Assuming that $\x_1,..,\x_N$ is an independently and identically distributed (iid) sample with a SIRV distribution on $\mathbb{C}^d$, the main focus of our study lies in the estimation of the constrained scatter matrix $\bSigma$ of the underlying distribution. Within this framework, we consider two kinds of robustness for the estimation of the scatter matrix :
\begin{itemize}
\item (R1)\ a robustness with respect to the distribution of the amplitude $\tau$ which is often heavy-tailed and will be considered as unknown.
\item (R2)\ a robustness with respect to a contamination in the observed sample: a fraction of the samples are outliers or come from a different distribution (see Section \ref{section_simu} for radar use cases of contamination).
\end{itemize}

\subsection{Prior art}
The estimation of covariance matrix of SIRV with or without structure has been the motivation of many works in the past, especially for radar applications. In the Gaussian framework, taking into account the structure of a covariance matrix has been shown to improve performance of both estimation \cite{stoica99} and target detection \cite{bose95}.\\
Non-Gaussian models of low grazing angle clutters involving SIRV has been early proposed in order to improve the estimation of the scatter matrix for example through non-linear transformations or cumulants as well as the performance of target detection in this framework \cite{farina87,gini97b,michels95}. These approaches however often consider that the law of the non-Gaussian amplitude $\tau$ is known.\\
Maronna proposed a class of Huber-type M-estimators of the scatter matrix $\bSigma$ that do not assume this knowledge (therefore robust in the sense of (R1)). They are defined as solution of the equation \cite{maronna76} :
\begin{equation}
\hat\bSigma = \frac{1}{N} \sum_{i=1}^N u(\x_i^*\hat\bSigma^{-1}\x_i)\x_i\x_i^*.
\end{equation}
The function $u$ has to satisfy some conditions for the estimator to be defined and consistent. A major drawback of these estimators is their non-invariance with respect to the distribution of the amplitude. For this sake, Tyler \cite{tyler87} proposed the estimator satisfying 
\begin{equation}
\hat\bSigma = \frac{d}{N} \sum_{i=1}^N \frac{\x_i\x_i^*}{\x_i^*\hat\bSigma^{-1}\x_i}.
\label{tyler}
\end{equation}
The function $u(x)=\frac{1}{x}$ does not satisfy the conditions of Maronna but Tyler has shown that the estimator solution of Eq. (\ref{tyler}) is well defined and consistent. It was furthermore shown to be a maximum likelihood estimator for normalized samples $\frac{\x_1}{\|\x_1\|},..,\frac{\x_N}{\|\x_N\|}$ (often called multivariate signs). Some authors rediscovered and studied this estimator in its complex version in the radar context \cite{conte02,gini02,pascal08}. These estimators however do not consider any structure on the matrix $\Sigma$.\\

In case of stationary signals, we have to take into account a Toeplitz intrinsic structure for the scatter matrix in the SIRV framework. It can be performed by computing the constrained maximum likelihood of normalized samples (or other statistical criterion) in the space of positive definite matrices with Toeplitz constraints or more specific structures (see \cite{bini14,soloveychik14,sun15,pailloux10,zhang13,wiesel15}). Our approach is slightly different in the sense that we propose a reparametrization of the scatter matrix and minimize a criterion based on the induced parameters.\\
Indeed, the autoregressive structure allows us to split the estimation of the matrix $\bSigma$ of size $d\times d$ into $d$ estimations of Toeplitz matrices of size $2\times 2$. This splitting corresponds to the so-called ``Burg technique" \cite{burg78}. Instead of estimating the covariance of the raw sample $\x_1,...,\x_N\in\mathbb{C}^d$, we iteratively define second-order samples in $\mathbb{C}^2$ whose theoretical covariance can be expressed in function of $\bSigma$ (see Section \ref{section_burg}).\\
This technique was originally proposed in the context of stationary Gaussian AR time series. Note that if we consider $\x$ as the trace of an AR process of order $M<d-1$, we add more structure on the matrix $\bSigma$ than the Toeplitz one. Actually, given the autocovariance taps $\mathbb{E}[x_1\overline{x_k}]$ for $k= 1,...,M$ with $M\leq d-1$, it is well known that the maximum entropy model pertaining to the vector $\x=(x_1,...,x_d)^T$ in $\mathbb{C}^d$ results as the complex Gaussian distribution in $\mathbb{C}^d$, whose covariance coincides with the autocovariance of size $d\times d$ of the trace of a Gaussian AR process of order $M$ (see \cite{barbaresco14,burg78,pavon13}).\\
The reparametrization of any Toeplitz covariance matrix by one real positive power parameter and $d-1$ complex coefficients (called reflection parameters; see Section \ref{section_burg}) underlying the Burg technique was also denoted by Trench (see \cite{trench64}). The Toeplitz structure is then a particular case for the class of autoregressive structures when $M=d-1$.\\

\subsection{Outline of the paper}

The Burg technique was initially expressed for the estimation of the autoregressive model of one process (hence one range case). However, it is possible to modify Burg estimates in order to combine multiple samples coming from different range cases (``segments'')  in the Gaussian context \cite{haykin82}. The first contribution of the following paper is the adaptation of this Multisegment Gaussian Burg estimation, that we will call Normalized Burg, to the SIRV context.\\


An alternative to Multisegment estimations is to take the mean of models coming from each estimated segment \cite{beex86}. Since reflection parameters do not depend on the amplitude realization (Burg technique separates the power parameter from reflection parameters), the robustness with respect to (R1) is ensured. However such estimators are not robust with respect to (R2) like Multisegment (Gaussian and Normalized) estimates. We propose then to study a geometrical method consisting in computing the median (instead of the mean) of AR parameters estimated from $\x_1,...,\x_N$ in both Euclidean and Riemannian context (see \cite{arnaudon13,barbaresco12,yang12}) and a refinement consisting in a 2-step estimation through a selection of the ``better" samples presented in \cite{aubry12,aubry14}. This will allow us to cumulate robustness with respect to heavy-tailed amplitudes (R1) and contamination (R2).\\

The paper is organized as follows. Section \ref{section_mixtures} presents the mixture of AR vector model. We introduce the Normalized Burg algorithm adapted to the case of SIRV models in Section \ref{section_burg}. We also present the average Burg estimators as well as its robust (with respect to (R2)) modifications in this Section. Finally, we illustrate the performances of the introduced estimators through some simulations of radar scenarios in Section \ref{section_simu}.

\section{Mixtures of autoregressive processes}
\label{section_mixtures}
Let $\x\in\mathbb{C}^d$ be the random vector coming from a mixture of stationary Gaussian AR random vectors. Then, $\x$  is characterized by the existence of a scalar random variable $\tau>0$ and a scatter matrix $\bSigma$ such that : 
\begin{equation}
\x = \tau \y
\end{equation}
where $\y\sim\mathcal{N}_d(0,\bSigma)$ is a complex Gaussian vector (called speckle) of covariance matrix $\bSigma$ independent of $\tau$ (called texture). $\bSigma$ is then defined up to a multiplicative constant due to the presence of $\tau$ (we can multiply $\bSigma$ and divide $\tau$ by the same positive constant without changing the vector $\x$). We will then consider in the following the constraint $\text{Tr}(\bSigma) = d$ (see \cite{gini02}).\\
As $\y$ comes from a stationary Gaussian AR process of order $M\leq d-1$, if we note $\y=(y_1, \dots, y_d)^T$, there exist $a_1^{(M)},...,a_M^{(M)}\in\mathbb{C}$ such that for $1\leq n \leq d$ :
\begin{equation}
y_n + \sum_{i=1}^M a^{(M)}_i y_{n-i} = b_n
\end{equation}
where $b_n$ is a complex standard Gaussian variable independent of $y_{n-1},..., y_{n-M}$ and with the convention $y_{-i}=0$ for all $i\geq 0$.\\
We can remark that $\x$ is also an AR process with non-Gaussian innovations : 
\begin{equation}
x_n + \sum_{i=1}^M a^{(M)}_i x_{n-i} = \tau b_n.
\end{equation}

\section{Burg algorithms}
\label{section_burg}
\subsection{Multisegment Gaussian Burg method applied to Gaussian process}
We first present the well-known Burg method for Gaussian AR vectors. All the definitions we introduce for the process $\y$ are still valid for $\x$.\\
Let define the autocorrelation function for $t\geq 0$ : $\gamma(t) = \mathbb{E}[y_{n+t}\overline{y_n}]$ for any $n$ such that the expectation has a sense. $\gamma$ is independent of $n$ because of the stationarity of $\y$. Moreover, the stationarity condition can be summed up by Yule-Walker equation :
\begin{equation}
\left(\begin{array}{ccc}
\gamma(0)  & \hdots & \gamma(M-1)\\
\overline{\gamma(1)}  & \hdots & \gamma(M-2) \\
\vdots & \vdots &  \vdots \\
\overline{\gamma(M-1)}  & \hdots & \gamma(0)
\end{array}\right)
\left(\begin{array}{cc}
a_M^{(M)}\\ \vdots \\ a_1^{(M)}
\end{array}\right)
= -
\left(\begin{array}{cc}
\gamma(M)\\ \vdots \\ \gamma(1)
\end{array}\right).
\end{equation}
Levinson algorithm inverts this equation by introducing the successive AR parameters $(a_k^{(m)})_{1\leq k\leq m}$ of order $1\leq m\leq M$  :
\begin{itemize}
\item Initialization : define $P_0 = \gamma(0)$ and
\begin{equation}
\left\{\begin{array}{l}
\mu_1 = a_1^{(1)} = - \frac{\gamma(1)}{P_0}\\
P_1 = P_0(1-|\mu_1|^2)
\end{array}\right..
\label{levinson0}
\end{equation}
\item For $1\leq m\leq M-1$
\begin{equation}
\left\{\begin{array}{l}
\mu_{m+1} = a_{m+1}^{(m+1)} = - \frac{\gamma(m+1)+\sum_{k=1}^m a_k{(m)}\gamma(m+1-k)}{P_m}\\
P_{m+1} = P_m(1-|\mu_m|^2)\\
\left(\begin{array}{c} a_1^{(m+1)}\\ \vdots\\ a_m^{(m+1)}\end{array}\right)
=\left(\begin{array}{c} a_1^{(m)}\\ \vdots\\ a_m^{(m)}\end{array}\right)
+\mu_{m+1}\left(\begin{array}{c} \overline{a}_m^{(m)}\\ \vdots\\ \overline{a}_1^{(m)}\end{array}\right)
\end{array}\right..
\label{levinson}
\end{equation}
\end{itemize}
This algorithm enhances the role of the parameters $(\mu_m)_{1\leq m\leq M}$, called reflection parameters, that are sufficient with $P_0$ to describe the AR vector $\y$. Note furthermore that the condition $P_0=1$ is equivalent to the aforementioned condition $\text{Tr}(\bSigma)=d$.\\
Instead of estimating the covariance matrix directly from the samples which does not guarantee the Toeplitz constraint, we estimate these reflection parameters adapted for AR random vectors (we will then use the bijection given by equations (\ref{levinson0}) and (\ref{levinson}) to recover an estimated covariance).\\
For this purpose, Burg proposed in the Gaussian framework to minimize an energy at each step $1\leq m\leq M$ :
\begin{equation}
U^{(m)} = \sum_{n=m+1}^d |f_m(n)|^2 + |b_m(n)|^2
\end{equation}
with $f_m$ and $b_m$ respectively the ``forward" and ``backward" errors defined for $m+1\leq n \leq d$ : 
\begin{equation}
\left\{\begin{array}{l}
f_m(n) = y_n + \sum_{k=1}^m a_k^{(m)}y_{n-k}\\
b_m(n) = y_{n-m} + \sum_{k=1}^m \overline{a}_k^{(m)}y_{n-m+k}
\end{array}\right..
\label{deferrors}
\end{equation}
Note that the definition of the errors is still valid for $m=0$. Thanks to Equation (\ref{levinson}), we can state for $m+2 \leq n \leq d$ :
\begin{equation*}
\left\{\begin{array}{l}
f_{m+1}(n) =f_{m}(n) +\mu_{m+1}b_m(n-1)\\
b_{m+1}(n) =b_{m}(n-1) +\overline{\mu_{m+1}}f_m(n)
\end{array}\right..
\end{equation*}

\begin{remark}
\label{remark_reg}
When there is no prior information on the model order, namely $M$, we should take it as high as possible, i.e. $M=d-1$. However, when $N$ is small, this choice could lead to a poor estimation of the reflection parameters even if the ``true'' model order is low. A classical way to solve this problem is to minimize the energy $U^{(m)}+\gamma C^{(m)}$ where $C^{(m)}$ corresponds to a spectral smoothness of the AR process and $\gamma$ tuned the compromise between regularization and estimation; see \cite{barbaresco96} for details on the regularization of Gaussian Burg estimators and \cite{decurninge14} for the regularized version of Normalized Burg defined hereafter.
\end{remark}
\vspace{0.5cm}
The estimation of the reflection parameters consists then in the solution of the minimization of the empirical energy for a sample $\x_1,..,\x_N$ : 
\begin{equation*}
\hat{U}^{(m)} = \sum_{i=1}^N \sum_{n=m+1}^d |f_{i,m}(n)|^2 + |b_{i,m}(n)|^2
\end{equation*}
where, for each $i$, $f_{i,m}$ and $b_{i,m}$ are the forward and backward errors for the sample $\y_i$. Knowing $\mu_1,...,\mu_m$, we have a closed-form expression for the estimate of $\mu_{m+1}$  (this estimator is called Multisegment Gaussian Burg see \cite{haykin82}) :
\begin{equation}
\hat\mu_{m+1}^{(gauss)} = \arg\min_{\mu_{m+1}} \hat{U}^{(m+1)} = -2 \frac{ \sum_{i=1}^N \sum_{n=m+2}^d f_{im}(n)\overline{b_{im}(n-1)}}{ \sum_{i=1}^N \sum_{n=m+2}^d |f_{im}(n)|^2 + |b_{im}(n-1)|^2}.
\label{mugauss}
\end{equation}

\subsection{Multisegment Normalized Burg method for non-Gaussian vectors}

We now consider the AR vector $\x$. The forward and backward errors are still defined by Equation (\ref{deferrors}). The estimator defined by Equation (\ref{mugauss}) applied for $\x$ will suffer from the disparity of the realizations of the scalar part $\tau$ which leads us to adapt the method by considering a different energy independent of the realizations of the texture $\tau$:
\begin{equation}
U^{(m+1)} = \sum_{n=m+2}^d \frac{|f_{m+1}(n)|^2 + |b_{m+1}(n)|^2}{|f_{m}(n)|^2 + |b_{m}(n-1)|^2}.
\end{equation}
The minimum of the empirical version of the previous energy is then : 
\begin{equation}
\hat\mu_{m+1} = -\frac{2}{N(d-m-1)}\sum_{i=1}^N \sum_{n=m+2}^d  \frac{\overline{b_{i,m}(n-1)} f_{i,m}(n)}{|f_{i,m}(n)|^2 + |b_{i,m}(n-1)|^2}.
\label{estimator1}
\end{equation}
The drawback is that $\hat\mu_{m+1}$ is not consistent. We can however correct the asymptotic bias:
\begin{proposition}
For $1 \leq m \leq M$ and $\hat{\mu}_m$ defined by Eq. (\ref{estimator1})
\begin{equation}
\hat\mu_{m} \as B_1(|\mu_m|) \frac{\mu_m}{|\mu_m|}
\end{equation}
with $B_1$ defined for $x>0$ by :
\begin{equation}
B_1(x) = \frac{1-x^2}{x}\left(\frac{\log(1-x)-\log(1+x)}{2x} + \frac{1}{1-x^2}\right).
\end{equation}
\end{proposition} 
\begin{proof}
This is an application of Theorem 1 of \cite{bausson07} applied for the vector $\left(\begin{array}{c}f_m(n)\\b_m(n-1)\end{array}\right)$. We apply the law of large numbers for the empirical sum $\hat\mu_m$ by noting that for $m\geq 0$ and $m+2 \leq n \leq d$ (Prop. 1 of \cite{dalhaus03} applied for $\{k_1,..,k_m\} = \{1,...,m\}$)
\begin{equation*}
\left\{\begin{array}{l}
\mathbb{E}[|f_{m}(n)|^2] = \mathbb{E}[|b_{m}(n-1)|^2] = P_m\\
\mathbb{E}[f_{m}(n)\overline{b_{m}(n-1)}|] = -P_m\mu_{m+1}
\end{array}\right..
\end{equation*}
\end{proof}

\begin{figure}[!t]
\centering
\includegraphics[width=2.5in]{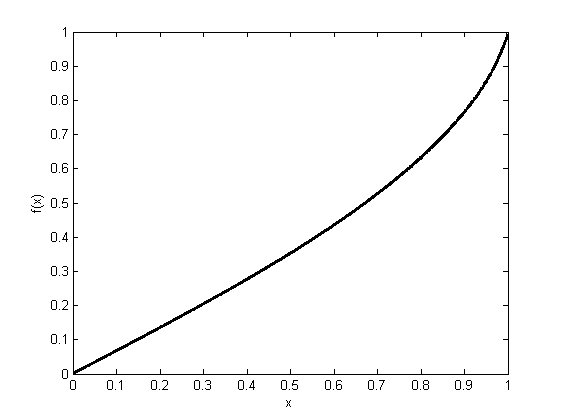}
\caption{Bias function $B_1$.}
\label{fig_bias}
\end{figure}

The consistent version of (\ref{estimator1}) is then : 
\begin{equation}
\hat\mu^{(u)}_{m} = B_1^{-1}(|\hat\mu_{m}|)\frac{\hat\mu_{m}}{|\hat\mu_{m}|}.
\label{normburg}
\end{equation}
$B_1^{-1}$ is not explicit but can be pre-computed on a grid for a gain of time (see Fig. \ref{fig_bias}).

\subsection{Average Burg estimators}
\subsubsection{Euclidean Mean Burg}
The Gaussian and Normalized Burg estimators presented above combine two processes:
\begin{itemize}
\item the iteration on the reflection parameters (with a propagation of errors from $\mu_1$ to $\mu_M$)
\item the averaging process with the $N$ spatial range cases
\end{itemize}
The Gaussian and Normalized Burg both perform the spatial averaging in the temporal iteration loop.\\
However, since spatial range cases may not be statistically homogeneous, it seems more robust to treat these two processes separately in order that the error due to the presence of outliers are not propagating in the iteration on the reflection parameters. For that purpose, since the realization of the amplitude parameter $\tau$ is shared amongst one range case, a Gaussian Burg estimation of the reflection parameters of each range case should be performed and a spatial ``average'' afterwards. Since the number of pulses $d$ may be small, the estimation of the reflection parameters for each range case $i$ (denoted hereafter by $\hat{\mu}_{m}^{(i)}$) may not be accurate, i.e. the variance of these local estimates will be high (it is comparable to the variance of $\x_i\x_i^*$ as an estimate of the covariance) but the spatial averaging counters this effect as detailed below.\\

With previous notations, a simple average estimator (defined e.g. in \cite{beex86}) is then (the superscript $.^{(i)}$ will refer to the $i$-th range case)
\begin{equation}
\hat\mu_{m+1} = \frac{1}{N}\sum_{i=1}^N \frac{-\sum_{n=m+2}^d  \overline{b_{i,m}(n-1)} f_{i,m}(n)}{\sum_{n=m+2}^d  1/2[|f_{i,m}(n)|^2 + |b_{i,m}(n-1)|^2]} := \frac{1}{N}\sum_{i=1}^N{\hat\mu_{m+1}^{(i)}}
\label{estimator2}
\end{equation}
Note that the iterative errors $b_{i,m}$ and $f_{i,m}$ are estimated only with the temporal samples $\x_i$ of the range case $i$. Hence, this is not a multisegment estimator since each reflection parameter $\hat\mu_{m+1}^{(i)}$ is estimated independently from the other range cases.\\
The average process will not affect the bias but will divide the variance by a factor $N$ since	
\begin{eqnarray*}
\mathbb{E}\left[\frac{1}{N}\sum_{i=1}^N\hat\mu^{(i)} - \mu\right] &=& \mathbb{E}[\hat\mu^{(1)} - \mu]\\
\text{and \ \ \ \ \ \ \ }\mathbb{E}\left[\left|\frac{1}{N}\sum_{i=1}^N\hat\mu^{(i)} - \mu\right|^2\right] &=& \frac{1}{N}\mathbb{E}[|\hat\mu^{(1)} - \mu|^2].
\end{eqnarray*}

Remark that the regularization evoked in Remark \ref{remark_reg} also induces a decreasing variance but an increasing bias (since some a priori knowledge on the spectral smoothness is introduced). For $N$ large enough, it is therefore not necessary to introduce bias since the presence of multiple spatial samples already decreases the variance of the estimators.\\


\subsubsection{Poincar\'e Mean Burg}
\label{section_poincaremean}
Eq. (\ref{estimator2}) corresponds to the Euclidean mean of Burg estimators of each (short) time series $\x_i$. In the space of positive definite matrices, Aubry et al \cite{aubry12} showed through simulations the superiority of non-Euclidean metric in terms of performance of the target detection. In the case of reflection parameters, a generalization to an arbitrary Riemannian geometry is also possible 
\begin{equation}
\hat\mu_{m+1} = \text{mean}\left( \frac{-\sum_{n=m+2}^d  \overline{b_{i,m}(n-1)} f_{i,m}(n)}{\sum_{n=m+2}^d  1/2[|f_{i,m}(n)|^2 + |b_{i,m}(n-1)|^2]}\right),
\label{estimator}
\end{equation}
where we classically call mean the so-called Fr\'echet mean defined as a minimizer for a certain distance $d(.,.)$
\begin{eqnarray*}
\text{mean}(\mu^{(1)},\mu^{(2)},...,\mu^{(N)}) &=& \arg\min_{|\mu|<1} \sum_{i=1}^N d(\mu,\mu^{(i)})^2.
\end{eqnarray*}
We will consider a Riemannian metric related to the statistical model parameterized by the reflection parameters. Indeed, a natural information geometry can be associated to any parametric model by defining a Riemannian metric through the Fisher information  matrix or its dual version (see \cite{arnaudon13,barbaresco12,barbaresco14,barbaresco15,yang12}), that we will consider here, defined by
\begin{equation*}
ds^2 = \sum_{i,j} \frac{\partial \phi}{\partial w_i \partial w_j}dw_idw_j
\end{equation*}
where $\w=(P_0,\mu_1,...,\mu_M)^T$ and $\phi$ denotes the entropy of the Gaussian AR vector $\y$
\begin{equation*}
\phi(P_0,\mu_1,...,\mu_M) = -\sum_{k=1}^{M} (M+1-k)\log(1-|\mu_k|^2)-(M+1)\log(\pi e. P_0)
\end{equation*}
which gives us:
\begin{equation*}
ds^2 = (M+1)\left(\frac{dP_0}{P_0}\right)^2+\sum_{k=1}^M (M+1-k)\frac{|d\mu_k|^2}{(1-|\mu_k|^2)^2}.
\end{equation*}
Fortunately, the geometry associated to the AR model reparametrized by reflection parameters is simple in the sense that the geodesics do not have cross products which justifies the separation of power and reflection parameters. For each of these latter, the natural metric is the Poincar\'e metric in the unit disc $D=\{z\in\mathbb{C}\text{  s.t.  }|z|<1 \}$ : 
\begin{equation}
ds^2 = \frac{|dz|^2}{(1-|z|^2)^2}.
\end{equation}
which leads to the following distance function²
\begin{equation*}
d_P(\mu^{(1)},\mu^{(2)}) = \frac{1}{2}\log\left(\frac{1+\delta}{1-\delta}\right) \text{\ \ \ with \ \ \ } \delta = \left|\frac{\mu^{(1)}-\mu^{(2)}}{1-\mu^{(1)}\overline{\mu^{(2)}}}\right|.
\end{equation*}
Note that we derived the Poincar\'e metric from the entropy of a Gaussian distribution. We ignored the non-Gaussian amplitude since we will not use in the following the power parameter but each reflection parameter. Moreover, we consider each reflection parameters independently since we do not want to mix the actual reflection parameters of the clutter (of low order) and the reflection parameters (of high order) whose value should be estimated to $0$.\\
If one wants to take into account the amplitude distribution in a SIRV setting, the related metric will depend on the fixed amplitude distribution (see e.g. \cite{calvo02}).\\

The mean estimators have the robustness (R1) since they are independent with respect to the texture $\tau$ but they fail to be robust with respect to (R2). This is the reason of our interest into the replacement of the mean by a median.

\subsubsection{Robust Euclidean and Poincar\'e Median Burg}

The Fr\'echet median of $N$ points in a Riemannian manifold is usually defined by:
\begin{equation*}
\text{median}(\mu^{(1)},\mu^{(2)},...,\mu^{(N)}) = \arg\min_{|\mu|<1} \sum_{i=1}^N d(\mu,\mu^{(i)}).
\label{estimatormed}
\end{equation*}

The computations of means and medians for Poincar\'e metric presented above are available in \cite{arnaudon13}. In the Euclidean framework, the median can be computed through Weiszfeld algorithm (see \cite{vardi00} for a modification of Weiszfeld algorithm that is convergent for all initial point), which is initialized with a point $z_0\in\mathbb{C}$ and for $t\geq 0$
\[
z_{t+1} = \frac{\sum_{i=1}^N \mu^{(i)} / |z_t-\mu^{(i)}|}{\sum_{i=1}^N 1 / |z_t-\mu^{(i)}|}.
\]

The Poincar\'e metric will favor estimates close to the center by penalizing the angle inhomogeneity of the reflection parameters of the local estimates $\mu_m^{(1)},...,\mu_m^{(N)}$. This is illustrated by Fig. \ref{fig_samples} in Section \ref{section_simu}.\\

\subsection{2-step procedures}
\label{section_2step}
The idea of 2-step procedure is to use a first estimation and to select the ``best'' samples in order to remove spatial samples containing potential outliers. The choice of the discarded samples is made according to the distance of the estimated reflection parameter of each range case to the robust estimate
\[
d(\mu_m^{(1:N)},\hat\mu_m)\leq \dots \leq d(\mu_m^{(N:N)},\hat\mu_m).
\]
We then keep the only $N/2$ closest range cases that will be supposed to be statistically homogeneous. Unlike \cite{aubry12} where choice of such ``secondary data'' is made according to the so-called generalized inner product (GIP) $\x_i\hat\Sigma^{-1}\x_i$, our criterion is based on the Euclidean or Riemannian distance between reflection parameter (see Algorithm \ref{algo_2step}). Indeed, the GIP criterion depends on the power realizations and is then not robust with respect to (R1).

\subsection{Algorithm summary}

\begin{algorithm}[H]
\caption{Generalized Burg-Levinson algorithm}
\begin{algorithmic}
\State $\textbf{Aim}$ : Estimation of the power and reflection parameters $(P_0,\mu_1,...,\mu_M)$
\State $\textbf{Input}$ : a sample of $N$ vectors $(\x_1,..,\x_N)$ in $\mathbb{C}^d$, the order of the autoregressive process $M$

\State $P_0 = \frac{1}{Nd}\sum_{i=1}^N\sum_{k=1}^d |x_{ik}|^2$
\State For $1\leq i \leq N$ and $1\leq n \leq d$ , $\left\{\begin{array}{l}f_{i,0}(n) = x_{in}\\b_{i,0}(n) = x_{in}\end{array}\right.$
\State $\textbf{for  } m=1...M$
\State  \hspace{0.5cm} Estimation of $\hat\mu_{m}$ from $f_{m-1}$ and $b_{m-1}$ (through e.g. Normalized Burg estimator)
\State  \hspace{0.5cm} $P_m=(1-|\hat\mu_m|^2)P_{m-1}$
\State \hspace{0.5cm} $\left(\begin{array}{c} a_1^{(m)}\\ \vdots\\ a_{m-1}^{(m)}\end{array}\right) =\left(\begin{array}{c} a_1^{(m-1)}\\ \vdots\\ a_{m-1}^{(m-1)}\end{array}\right) +\hat\mu_{m}\left(\begin{array}{c} \overline{a}_{m-1}^{(m-1)}\\ \vdots\\ \overline{a}_1^{(m-1)}\end{array}\right) $
\State \hspace{0.5cm} $a_m^{(m)}=\hat\mu_m$
\State \hspace{0.5cm} Forward and backward errors for $1\leq i \leq N$ and $m+1\leq n \leq d$ ,  
\State \hspace{0.5cm}$\left\{\begin{array}{l}
f_{i,m}(n) =f_{i,m-1}(n) +\hat\mu_{m}b_{i,m-1}(n-1)\\
b_{i,m}(n) =b_{i,m-1}(n-1) +\overline{\hat\mu_{m}}f_{i,m-1}(n)
\end{array}\right.$
\State $\textbf{end}$
\end{algorithmic}
\label{algo_gburg}
\end{algorithm}

\begin{algorithm}
\caption{Normalized Burg}
\begin{algorithmic}
\State $\textbf{Aim}$ : Estimation of the $m$-th coefficient of reflection $\hat\mu_{m+1}$
\State $\textbf{Input}$ : forward and backward errors $f_{i,m}(n)$ and $b_{i,m}(n)$
\State $z = -\frac{2}{N(d-m-1)}\sum_{i=1}^N \sum_{n=m+2}^d  \frac{\overline{b_{i,m}(n-1)} f_{i,m}(n)}{|f_{i,m}(n)|^2 + |b_{i,m}(n-1)|^2}$
\State $B_1 = x\mapsto\frac{1-x^2}{x}\left(\frac{\log(1-x)-\log(1+x)}{2x} + \frac{1}{1-x^2}\right)$
\State $\hat\mu_{m+1} = B_1^{-1}(|z|)\frac{z}{|z|}$
\end{algorithmic}
\label{algo_nb}
\end{algorithm}

\begin{algorithm}
\caption{2-step Median Burg}
\begin{algorithmic}
\State $\textbf{Aim}$ :  Estimation of the reflection parameters $(\mu_1,...,\mu_M)$
\State $\textbf{Input}$ :  a sample of $N$ vectors $(\x_1,..,\x_N)$ in $\mathbb{C}^d$, the order of the autoregressive process $M$
\State For each range case $1\leq i\leq N$ compute the Gaussian Burg estimates $(\hat\mu_1^{(i)},...,\hat\mu_M^{(i)})$ (Eq. \ref{mugauss})
\State \textbf{for} m=1:M
\State \hspace{0.5cm} $\hat\mu_{m}^0 = \text{median}(\hat\mu_m^{(1)},...,\hat\mu_m^{(N)})$
\State \hspace{0.5cm} Order the reflection parameters with respect to the distance to $\hat\mu_{m}^0$
\State \hspace{0.5cm}  $d(\hat\mu_m^{(1:N)},\hat\mu_m^0)\leq...\leq d(\hat\mu_m^{(N:N)},\hat\mu_m^0)$
\State \hspace{0.5cm}  $\hat\mu_{m} = \text{median}(\hat\mu_m^{(1:N)},...,\hat\mu_m^{(N/2:N)})$
\State \textbf{end}
\end{algorithmic}
\label{algo_2step}
\end{algorithm}

\section{Simulations}
\label{section_simu}
\subsection{Simulated scenario}

As an illustration of  the performances of the defined algorithms, we modelize $N$ range cells of a clutter burst response $\z_1,...,\z_N$ through independent realization of the following random vector
\begin{equation}
\z_i\eqlaw \x_i+\w_i,
\label{model_clutter}
\end{equation}
where
\begin{itemize}
\item $\x\eqlaw \tau \y\in\mathbb{C}^d$ is a scale mixture of AR vectors; let us recall that $\tau$ is the texture and $\y$ the speckle (see Section \ref{section_mixtures}).
\item $\w\in\mathbb{C}^d$ is a white noise representing the thermal noise.
\end{itemize}
We choose a Weibull texture for the model of $\tau$ (considered for example in \cite{chong10}) for its adequacy with sea and ground clutters (see also \cite{conte04} for a validation of Weibull distribution on real data). A gamma-distributed texture corresponds to a K-distributed clutter which has been often proposed in the literature \cite{gini97,watts85}. However, we choose the Weibull one in order to model heavy-tailed clutters. We recall the expression of the density for a Weibull distribution : 
\begin{equation}
\text{for $x\geq 0$,   }\ \ \ \ \ \ \  f_{\tau}(x) = \frac{\nu}{\sigma} \left(\frac{x}{\sigma}\right)^{\nu-1} e^{-(x/\sigma)^{\nu}}.
\end{equation}
The scale parameter $\sigma$ is taken such that $\mathbb{E}[\tau]=\sigma\Gamma(1+1/\nu)$ is the desired clutter power whereas $\nu$ (taken equal to $0.6$ in order to modelize strong clutters) is the shape parameter representing the disparity of the distribution. We take $N=64$ samples and a speckle built from an AR vector of order 1 of parameter $\mu_1$ and of dimension $d$. An AR$(1)$ approximates a radar ground clutter or a wind clutter with a single Doppler frequency.\\
The Riemannian mean error (RME) for $N_{MC}$ estimations is a natural error metric in the space of positive definite matrices thanks to its affine invariance:
\begin{equation}
\text{RME} = \frac{1}{N_{MC}} \sum_{i=1}^{N_{MC}} \left\|   \log\left(\left(\hat\bSigma_i\right)^{-1/2} \bSigma_0\left(\hat\bSigma_i\right)^{-1/2}\right)\right\|_F
\end{equation}
where $\|.\|_F$ is the Frobenius norm. We will compare the following estimators of the scatter matrix :
\begin{itemize}
\item \textbf{(Multisegment) Gaussian Burg} : given by Eq. (\ref{mugauss}).
\item \textbf{Fixed Point} (FP) : the M-estimator proposed by Tyler \cite{tyler87}.
\item \textbf{(Multisegment) Normalized Burg} : estimator given by Equation (\ref{normburg}).
\item \textbf{Euclidean/Poincar\'e Mean Burg} : estimator given by Eq. (\ref{estimator2}) and (\ref{estimator}).
\item \textbf{Euclidean/Poincar\'e Median Burg} : estimator given in Section \ref{estimatormed}.
\item \textbf{2-step Euclidean/Poincar\'e Median Burg} : Algorithm \ref{algo_2step} by using Poincar\'e or Euclidean median and distances.
\item \textbf{2-step Fixed Point} : 2-step procedure (Section \ref{section_2step}) for Fixed Point algorithm with a selection of secondary data performed according to the likelihood of normalized samples.
\end{itemize}
The order of the above Burg estimators is taken to be maximal, i.e. $M=d-1$. Since the order of the simulated AR vector is $1$, this illustrates the robustness of the approach with respect to the choice of the order.\\
The estimation and detection performances will be compared to the following approaches: 
\begin{itemize}
\item a classical \textbf{OS-CFAR} used together with a Hamming window applied on outputs of Doppler Filters Bank; see for example \cite{rohling83}.
\item \textbf{Ideal} detection: we assume that the scatter matrix is known and use it for the test of Section \ref{sec_detection}. This constitutes a best-case performance benchmark for detection performances.
\end{itemize}

\subsection{Estimation quality}
\subsubsection{Robustness with respect to non-Gaussian amplitude}

Every tested estimator at the exception of Gaussian Burg is independent to the amplitude realizations, then we present a comparison of the estimation quality between Gaussian Burg and Normalized Burg with respect to the Weibull shape parameter in Table \ref{table_error0}.\\
Note that the Gaussian distribution corresponds to the limit case $\nu\rightarrow \infty$. This is the reason of the good behavior of Gaussian Burg estimates in Table \ref{table_error0} for large $\nu$. Moreover, as expected, the Gaussian Burg is largely outperformed by its Normalized version when the texture is sub-exponential ($\nu<1$), i.e. heavy-tailed.\\

\subsubsection{Influence of the number of pulses per range case}
In Table \ref{table_error1}, the superiority in terms of performance of the Normalized Burg with respect to Euclidean and Poincar\'e Mean Burg can be explained by the bias of the latter that is important when $d$ is small. Indeed, the bias of Mean Burg algorithms does not depend on $N$ contrary to Normalized Burg for which the performances would have been the same if we had considered a single temporal sequence of length $dN$. The same conclusion for the Multisegment estimate in the Gaussian case was given in \cite{waele00}.\\
The precision of Normalized Burg is deteriorated when $d$ increases since the order of the estimated autoregressive model (equal to $d-1$) increases with $d$. For large $d$, it is then useful to consider Mean Burg. However, since we are interested in contaminated scenarios where $d$ is small, we will restrict ourselves to the case $d=12$ it in the following.

\begin{table}[!t]
\center
\caption{Riemannian Mean Error for an AR(1) ($\mu_1=0.9$, $d=8$)} \begin{tabular}{|c|c|c|c|c|c|c|}
   \hline
&\parbox{3.5cm}{\centerline{Normalized Burg}} & \parbox{3.5cm}{\centerline{Gaussian Burg}} \\
\hline
$\nu=0.1$ & $0.42$ & $3.88$\\
$\nu=0.5$ & $0.42$ & $1.49$\\
$\nu=1$ & $0.42$ & $0.73$\\
$\nu=2$ & $0.42$ & $0.48$\\
$\nu=3$ & $0.42$ & $0.41$\\
$\nu=10$ & $0.42$ & $0.36$\\
\hline
\end{tabular}
\label{table_error0}
\end{table}

\begin{table}[!t]
\center
\caption{Riemannian Mean Error for an AR(1) ($\mu_1=0.9$, $\nu$ has no impact)} \begin{tabular}{|c|c|c|c|}
   \hline
&\parbox{3.5cm}{\centerline{Normalized Burg}} & \parbox{3.5cm}{\centerline{Euclidean Mean Burg}} & \parbox{3.5cm}{\centerline{Poincar\'e Mean Burg}} \\
\hline
$d=8$ & $0.42$ & $0.77$ & $ 0.76$\\
$d=16$ & $0.44$ & $0.63$ & $0.63$\\
$d=32$ & $0.47$ & $0.55$ & $0.55$\\
$d=64$ & $0.50$ & $0.49$ &  $0.49$\\
\hline
\end{tabular}
\label{table_error1}
\end{table}

\subsubsection{Estimation quality illustrated through a transition scenario}
\label{section_est}
\begin{table}[!t]
\center
\caption{Riemannian Mean Error for an AR(1) ($\mu_1=0.9$, $d=12$); these errors are independent of the shape parameter of the Weibull texture. We progressively increase the number of contaminating range cells (outliers).} \begin{tabular}{|c|c|c|c|c|c|c|}
   \hline
&\parbox{2.5cm}{\centerline{Normalized Burg}}  & \parbox{2.5cm}{Euclidean\\ Median Burg} & \parbox{2.5cm}{{2-step\\Euclidean  Median Burg}} & \parbox{2.5cm}{Poincar\'e\\ Median Burg} & \parbox{2.3cm}{\centerline{FP}} \\
\hline
$0$ outlier & $\bf 0.42$  &  $0.57$            & $1.01$ & $0.78$ & $0.70$ \\
$5$ outliers & $1.08$  & $\bf 0.64$                 & $0.98$ & $0.90$ & $0.78$ \\
$10$ outliers & $2.03$ & $\bf 0.77$               & $0.92$ & $1.04$ & $1.30$ \\
$20$ outliers & $3.17$ & $1.10$               & $\bf 0.87$ & $1.41$ & $2.70$ \\
$30$ outliers & $3.77$ & $1.96$               & $\bf 0.87$ & $2.56$ & $3.71$ \\
\hline
\end{tabular}
\label{table_error3}
\end{table}

\begin{table}[!t]
\center
\caption{Riemannian Mean Error for an AR(1) ($\mu_1=0.3$, $d=12$); these errors are independent of the shape parameter of the Weibull texture. We progressively increase the number of contaminating range cells (outliers).} \begin{tabular}{|c|c|c|c|c|c|c|}
   \hline
&\parbox{2.5cm}{\centerline{Normalized Burg}}  & \parbox{2.5cm}{Euclidean\\ Median Burg} & \parbox{2.5cm}{{2-step\\ Euclidean Median Burg}} & \parbox{2.5cm}{Poincar\'e\\ Median Burg} & \parbox{2.3cm}{\centerline{FP}} \\
\hline
$0$ outlier &0.34  & 0.38           & 0.90 &\bf 0.30 & 0.70 \\
$5$ outliers & 0.35  &  0.38            & 0.85 &\bf 0.33 & 0.70 \\
$10$ outliers & 0.41  &  0.41            & 0.85 &\bf 0.37 & 0.70 \\
$20$ outliers & 0.59  &  0.51            & 0.85 &\bf 0.49 & 0.76 \\
$30$ outliers & 0.71  &  0.67            & 0.97 &\bf 0.64 & 0.87 \\
\hline
\end{tabular}
\label{table_error4}
\end{table}

\begin{figure}
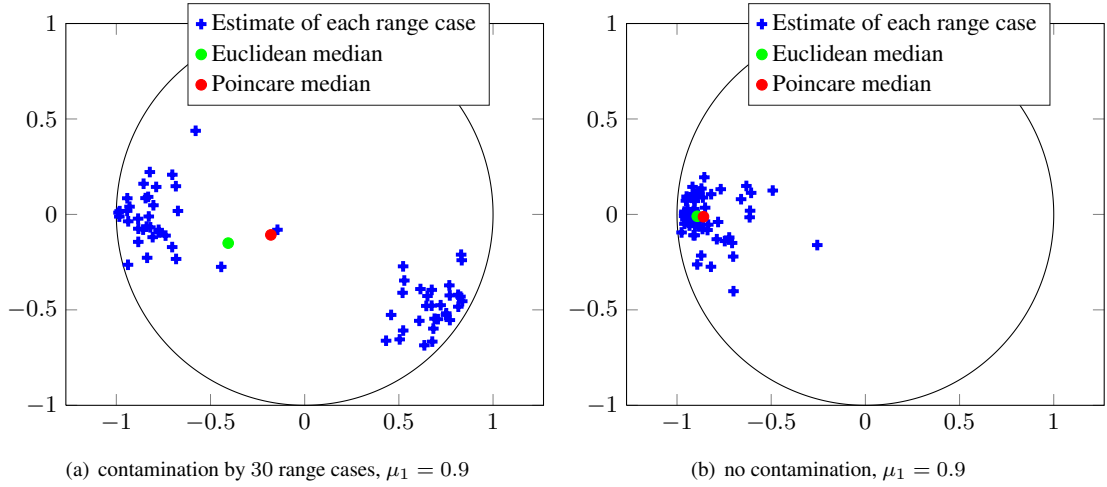

  \centering
   \subfigure[contamination by $30$ range cases, $\mu_1=0.9$]{ \small\input{samples_09.tikz}}
   \subfigure[no contamination, $\mu_1=0.9$]{ \small\input{samples_nc_09.tikz}}
 \caption{Estimated first coefficient of reflection for each range and their Riemannian and Euclidean medians in case of a contamination by $30$ range cases.}
 \label{fig_samples}
\end{figure}



\begin{figure*}[!t]
\centering
\subfigure[Simulated spectra]{\label{fig_clutterchange}\includegraphics[width=1.85in,height=1.2in]{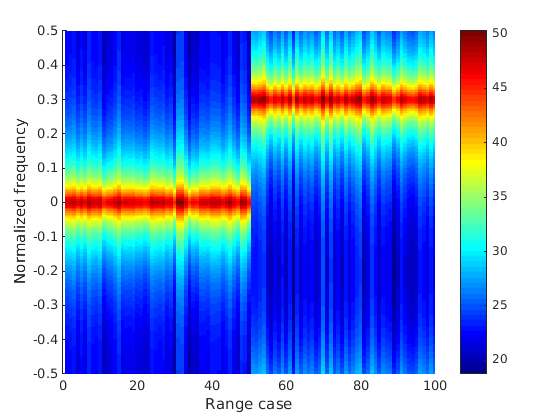}}
\subfigure[Normalized Burg]{\includegraphics[width=1.85in,height=1.2in]{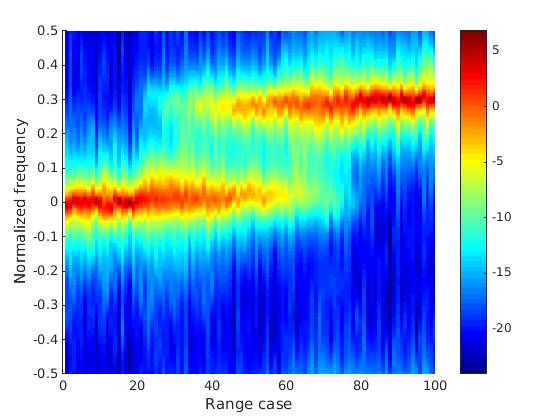}}
\subfigure[Fixed Point]{\includegraphics[width=1.85in,height=1.2in]{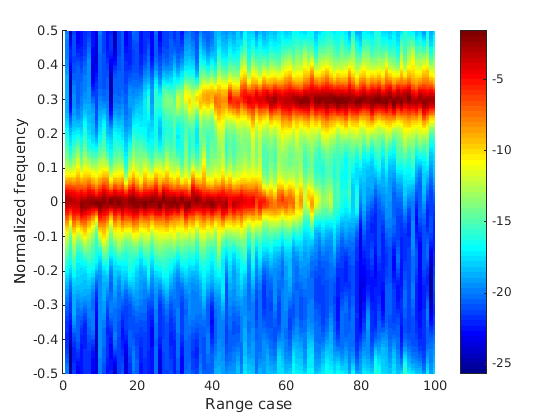}}\\
\subfigure[OS-CFAR]{\includegraphics[width=1.85in,height=1.2in]{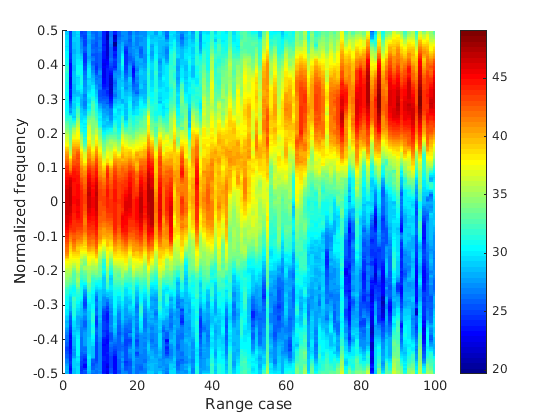}}
\subfigure[Poincar\'e Median Burg]{\includegraphics[width=1.85in,height=1.2in]{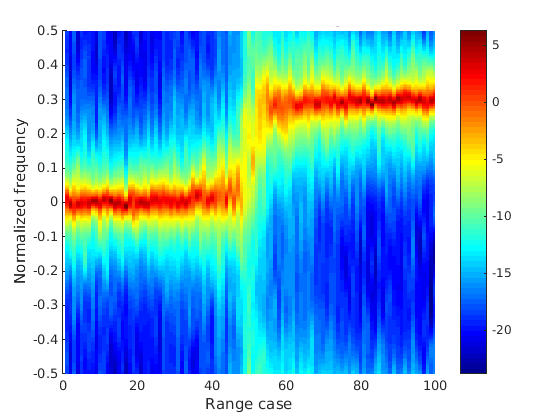}}
\subfigure[Euclidean Median Burg]{\includegraphics[width=1.85in,height=1.2in]{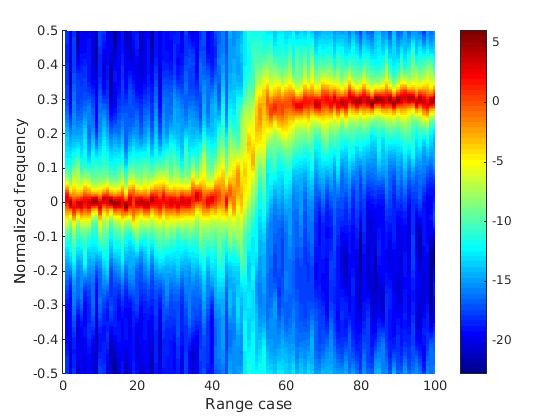}}\\
\subfigure[2-step Poincar\'e Median Burg]{\includegraphics[width=1.85in,height=1.2in]{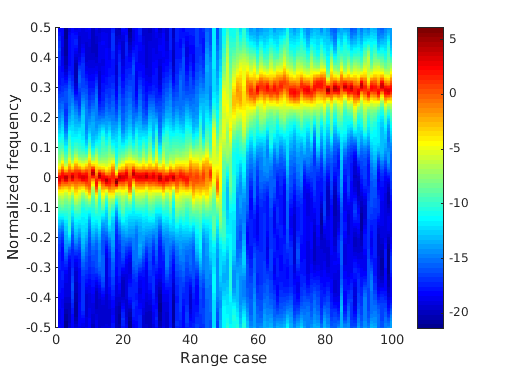}}
\subfigure[2-step Euclidean Median Burg]{\includegraphics[width=1.85in,height=1.2in]{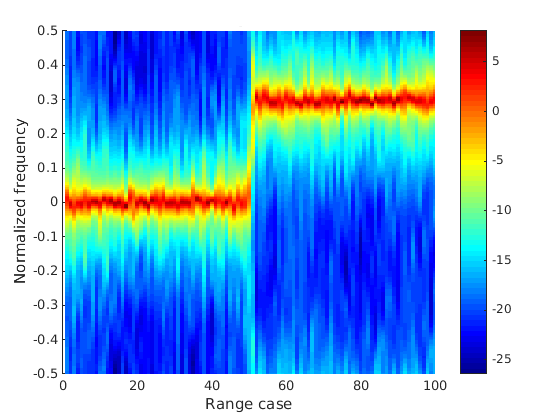}}
\subfigure[2-step Fixed Point]{\label{fig2stepfp}\includegraphics[width=1.85in,height=1.2in]{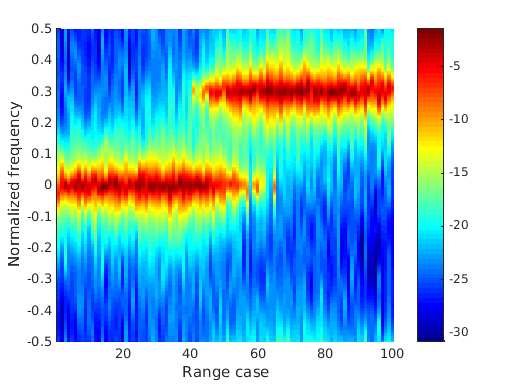}}
\caption{Estimated and simulated spectra for $100$ range cases (The x-axis corresponds to the range dimension whereas y-axis indicates the normalized frequency).}
\label{fig5}
\end{figure*}

We summarize in Table \ref{table_error3} and \ref{table_error4} the estimation errors for two scenarios ($\mu_1=0.9$ and $\mu_1=0.3$).
\begin{itemize}
\item In the non contaminated case, Normalized Burg show better accuracy than FP: taking into account the Toeplitz structure of the scatter matrix then improves the estimation quality.
\item When the spectrum is flatter ($|\mu_1|=0.3$), the Poincar\'e metric that favors the small coefficients is slightly more efficient than the Euclidean one. Indeed, the reflection coefficients of high order are then closer to $0$. For reflection parameters of high modulus however, this behavior makes Poincar\'e metric less efficient especially in the contaminated cases (see Fig. \ref {fig_samples}).
\item The 2-step procedure drastically increases the estimation quality in the case of a strong contamination and a strong correlation ($\mu_1=0.9$). In that case, the outliers coming from the ``true'' distribution are well separated from the correlation samples coming from the perturbing distribution and the 2-step procedure can then easily separate the two parts of the sample. Otherwise, when $|\mu_1|$ is low, this separation is less clear. 
\item The surprising decrease of the error in the 2-step procedure when the amount of contamination increases may be explained by the fact that the higher order reflection coefficient are better and better estimated thanks to the diversity brought by the contamination.
\end{itemize}

In Fig. \ref{fig5}, we illustrate the robustness of each estimator through a clutter transition. We considered a scenario where range cases $1$ to $50$ are simulated through an AR(1) of parameter $\mu_1=0.9$ and range cases $51$ to $100$ are simulated with an AR(1) of parameter $\mu_1=0.9e^{0.3\times 2i\pi}$ (Fig. \ref{fig_clutterchange}). For each test range case $\x_i$ ($i$ is represented in the $x$-axis), the represented Doppler spectrum results from the estimated covariance of the $N=64$ neighbor cells $\x_{i-32},...,\x_{i-1},\x_{i+1},...,\x_{i+32}$. In order to control edge effects, we consider only available neighbor cells for $32$ first and $32$ last cells.
\begin{itemize}
\item For the non robust estimators (namely Normalized Burg and Fixed-Point) the estimated spectra have two frequencies for cases around the transition which is not the case for the other estimators. 
\item For the robust estimators, the number of range cases where the estimated spectrum is not accurate is respectively for OS-CFAR, Poincar\'e Median Burg, Euclidean Median Burg and their respective 2 step versions of $16$, $8$, $8$ , $8$ and $2$ cases. With this property, the detection of a target with a normalized frequency of $0.3$ is possible in an area close to the transition for robust estimators. 
\item The 2-step procedure alone is not sufficient  if the first estimation is not robust: this is illustrated by Fig. \ref{fig2stepfp} where the secondary data selection after a first Fixed-Point estimation is not sufficient to separate the two clutter frequencies.
\item It can be observed that the spectra of OS-CFAR show a frequency resolution worse than its competitors. This is due to the low number of pulses ($d=12$) for each range cells which is responsible for the low number of filters in OS-CFAR.
\end{itemize}

\subsubsection{Estimation quality for a sea clutter scenario}

\begin{figure*}[!t]
\centering
\subfigure[Simulated spectra of neighbor range cases]{\label{fig_sweep}\includegraphics[width=2.8in]{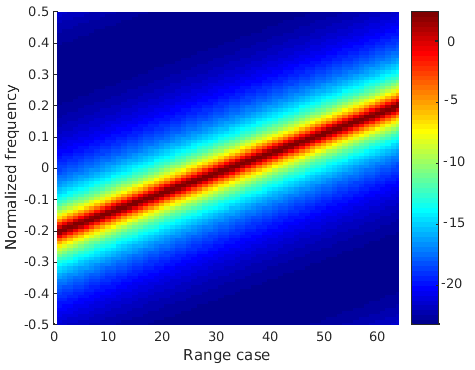}}\\
\subfigure[Estimated ``mean'' spectra of non-robust estimators]{\label{fig_sweep3}\scriptsize\input{sc1_spectres2.tikz}}
\subfigure[Estimated ``mean'' spectra of robust estimators]{\label{fig_sweep2}\scriptsize\input{sc1_spectres.tikz}}
\caption{Simulated and estimated spectrum in a sea clutter typical scenario in a Gaussian context; the simulated autoregressive process is an AR(1) with $\mu_1=0.9$.}
\label{figcluttersea}
\end{figure*}

\begin{figure*}[!t]
\centering
\subfigure[Simulated spectra of neighbor range cases]{\label{fig_sweepb}\includegraphics[width=2.8in]{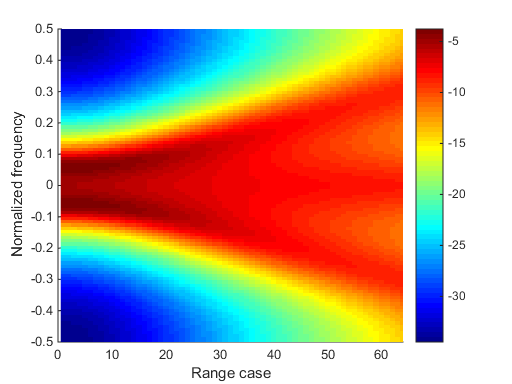}}\\
\subfigure[Estimated ``mean'' spectra of non-robust estimators]{\label{fig_sweep2b}\scriptsize\input{sc2_spectres.tikz}}
\subfigure[Estimated ``mean'' spectra of robust estimators]{\label{fig_sweep3b}\scriptsize\input{sc2_spectres2.tikz}}
\caption{Simulated and estimated spectrum in a sea clutter typical scenario in a Gaussian context; the simulated autoregressive process is an AR(3).}
\label{figcluttersea2}
\end{figure*}

In Fig. \ref{figcluttersea}-\ref{figcluttersea2}, we simulate a scenario encountered when we face sea clutter, namely the position of the ``peak'' in the spectra (respectively the spectral width for Fig. \ref{figcluttersea2}) of neighbor range cases is not stable and can be drifting (Fig. \ref{fig_sweep}). Ideally, the estimated ``mean'' spectra should correspond to a mean behavior, spectrally speaking: the position of the peak should be the mean of the neighbor peaks as well as the spectral width.
\begin{itemize}
\item In Fig. \ref{figcluttersea}, since Fixed Point and Normalized Burg estimators take into account every neighbor case with the same weights, the estimated spectrum is wider, this width representing the incertitude on the position of the peak. On the other-hand, the median-based estimators are only dependent on the considered Riemannian geometry in the space of reflection parameters. Indeed, with our choice of geometry, the more diversity in the parameters, the lower the absolute value of the median of these parameters and then the larger the spectrum. Since the 2-step Euclidean Median estimator takes into account less neighbor cases, the diversity is weaker and then, the accuracy of the estimated spectrum is higher. On the other hand, 2-step Poincar\'e Median estimator is sensitive to highest order reflection parameters that are more biased if we consider less range cases; this effect is moreover not sufficiently compensated by the first reflection parameters.
\item Similarly, in Fig. \ref{figcluttersea2}, robust estimators estimate more accurately the width of the ``average'' spectrum while non-robust estimators under-estimate it.
\end{itemize}

\subsection{Detection quality}
\label{sec_detection}

We will now compare the estimators through their detection performances (for the sake of clarity, we restrain ourselves to Normalized Burg, 2-step Burg estimators, Fixed-Point and OS-CFAR). For that purpose, we assume that a cell under test is spatially surrounded by $N$ neighbor cells sharing the same distribution or not (depending on the scenario). We suppose that a target is present in the cell under test. As for the target model, we consider small targets in the sense that they are present in only one range cell which burst response is simulated by
\begin{equation}
\z\eqlaw \alpha\p + \x+\w,
\end{equation}
where $\alpha$ represents the target power,  $\p=(1, e^{2i\pi f_D},...,e^{2i\pi (d-1)f_D})^T$, $f_D$ the normalized Doppler frequency and $\x$ and $\w$ are defined as in Eq. (\ref{model_clutter}).\\
Several test statistics dedicated to target detection in a non-Gaussian environment have been proposed in the radar literature. A detector classically used is the GLRT (Generalized Likelihood Ratio Test), also called ANMF (see e.g. \cite{gini97}). A variant of the GLRT detector has been proposed in \cite{alfano04} in the context of targets spread over several range cells present in an homogeneous or inhomogeneous clutter modelized as an autoregressive process.\\
We will use in the sequel the GLRT statistics and propose to compare its performances with a new geometrical detector based on the geometry of reflection parameters.

\subsubsection{GLRT detector}
Denoting by $\hat\bSigma$ one estimator of the scatter matrix of the neighbor cells, the GLRT detector is defined by
\begin{equation}
\text{GLRT}(\z) = \max_{\theta\in[-0.5;0.5[} \frac{|\p(\theta)^* \hat\bSigma^{-1}\z|^2}{(\z^*\hat\bSigma^{-1}\z)(\p(\theta)^*\hat\bSigma^{-1}\p(\theta))}
\end {equation}
with $\p(\theta)=(1, e^{2i\pi\theta},...,e^{2i\pi (d-1)\theta})^T$ the steering vector and $\z\in\mathbb{C}^d$ the data of the cell under test. We compute the test threshold such that the probability of false alarm is set to $10^{-3}$ and compare the probabilities of detection with the classical OS-CFAR test \cite{rohling83}.\\


%


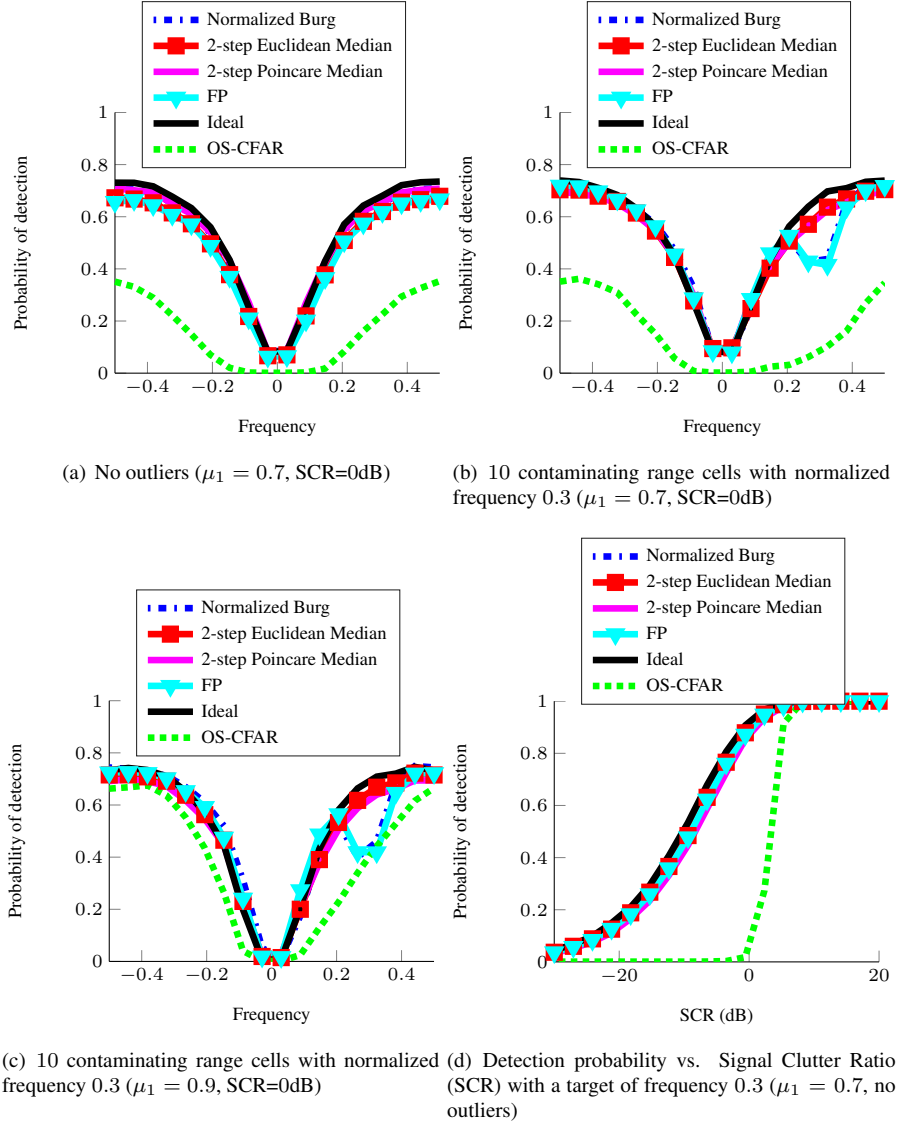
\begin{figure*}[!t]
\centering
\subfigure[No outliers ($\mu_1=0.7$, SCR=$0$dB)]{\scriptsize
%
%
\definecolor{mycolor1}{rgb}{1.00000,0.00000,1.00000}%
\definecolor{mycolor2}{rgb}{0.00000,1.00000,1.00000}%
\begin{tikzpicture}

\begin{axis}[%
width=1.7in,
height=1.365625in,
at={(0.758333in,0.48125in)},
scale only axis,
xmin=-0.5,
xmax=0.5,
xlabel={Frequency},
ymin=0,
ymax=1,
ylabel={Probability of detection},
title style={font=\bfseries},
axis x line*=bottom,
axis y line*=left,
legend style={at={(0.082602,0.78622)},anchor=south west,legend cell align=left,align=left,draw=white!15!black}
]
\addplot [color=blue,dash pattern=on 1pt off 3pt on 3pt off 3pt,line width=2.5pt]
  table[row sep=crcr]{%
-0.5	0.705\\
-0.441176470588235	0.7046\\
-0.382352941176471	0.6942\\
-0.323529411764706	0.6502\\
-0.264705882352941	0.6168\\
-0.205882352941176	0.536\\
-0.147058823529412	0.4206\\
-0.0882352941176471	0.2468\\
-0.0294117647058824	0.0814\\
0.0294117647058824	0.0884\\
0.0882352941176471	0.2398\\
0.147058823529412	0.4154\\
0.205882352941177	0.5488\\
0.264705882352941	0.62\\
0.323529411764706	0.6522\\
0.382352941176471	0.69\\
0.441176470588235	0.706\\
0.5	0.7096\\
};
\addlegendentry{Normalized Burg};

\addplot [color=red,solid,line width=2.5pt,mark=square,mark options={solid}]
  table[row sep=crcr]{%
-0.5	0.6726\\
-0.441176470588235	0.6678\\
-0.382352941176471	0.6534\\
-0.323529411764706	0.6148\\
-0.264705882352941	0.573\\
-0.205882352941176	0.496\\
-0.147058823529412	0.3778\\
-0.0882352941176471	0.2174\\
-0.0294117647058824	0.0668\\
0.0294117647058824	0.0702\\
0.0882352941176471	0.2196\\
0.147058823529412	0.377\\
0.205882352941177	0.509\\
0.264705882352941	0.5824\\
0.323529411764706	0.62\\
0.382352941176471	0.6556\\
0.441176470588235	0.6654\\
0.5	0.6784\\
};
\addlegendentry{2-step Euclidean Median};

\addplot [color=mycolor1,solid,line width=2.5pt]
  table[row sep=crcr]{%
-0.5	0.7086\\
-0.441176470588235	0.7016\\
-0.382352941176471	0.6928\\
-0.323529411764706	0.6504\\
-0.264705882352941	0.6102\\
-0.205882352941176	0.5396\\
-0.147058823529412	0.4306\\
-0.0882352941176471	0.2712\\
-0.0294117647058824	0.0812\\
0.0294117647058824	0.0852\\
0.0882352941176471	0.2696\\
0.147058823529412	0.4276\\
0.205882352941177	0.5506\\
0.264705882352941	0.6228\\
0.323529411764706	0.6552\\
0.382352941176471	0.6926\\
0.441176470588235	0.7032\\
0.5	0.7078\\
};
\addlegendentry{2-step Poincare Median};

\addplot [color=mycolor2,solid,line width=2.5pt,mark=triangle,mark options={solid,rotate=180}]
  table[row sep=crcr]{%
-0.5	0.6568\\
-0.441176470588235	0.664\\
-0.382352941176471	0.6442\\
-0.323529411764706	0.6072\\
-0.264705882352941	0.5706\\
-0.205882352941176	0.488\\
-0.147058823529412	0.3694\\
-0.0882352941176471	0.2098\\
-0.0294117647058824	0.0632\\
0.0294117647058824	0.065\\
0.0882352941176471	0.2018\\
0.147058823529412	0.3646\\
0.205882352941177	0.5054\\
0.264705882352941	0.5738\\
0.323529411764706	0.618\\
0.382352941176471	0.6512\\
0.441176470588235	0.6602\\
0.5	0.6666\\
};
\addlegendentry{FP};

\addplot [color=black,solid,line width=2.5pt]
  table[row sep=crcr]{%
-0.5	0.7306\\
-0.441176470588235	0.7302\\
-0.382352941176471	0.7166\\
-0.323529411764706	0.677\\
-0.264705882352941	0.6332\\
-0.205882352941176	0.56\\
-0.147058823529412	0.4344\\
-0.0882352941176471	0.254\\
-0.0294117647058824	0.0812\\
0.0294117647058824	0.0852\\
0.0882352941176471	0.2538\\
0.147058823529412	0.43\\
0.205882352941177	0.5712\\
0.264705882352941	0.6426\\
0.323529411764706	0.6804\\
0.382352941176471	0.721\\
0.441176470588235	0.7322\\
0.5	0.735\\
};
\addlegendentry{Ideal};

\addplot [color=green,dotted,line width=2.5pt]
  table[row sep=crcr]{%
-0.5	0.351\\
-0.441176470588235	0.331\\
-0.382352941176471	0.2894\\
-0.323529411764706	0.2214\\
-0.264705882352941	0.1488\\
-0.205882352941176	0.0702\\
-0.147058823529412	0.0206\\
-0.0882352941176471	0.0026\\
-0.0294117647058824	0.0012\\
0.0294117647058824	0.0016\\
0.0882352941176471	0.0028\\
0.147058823529412	0.0176\\
0.205882352941177	0.0832\\
0.264705882352941	0.1604\\
0.323529411764706	0.2256\\
0.382352941176471	0.2938\\
0.441176470588235	0.3262\\
0.5	0.3532\\
};
\addlegendentry{OS-CFAR};

\end{axis}
\end{tikzpicture}%
}
\subfigure[$10$ contaminating range cells with normalized frequency $0.3$ ($\mu_1=0.7$, SCR=$0$dB)]{\scriptsize
%
%
\definecolor{mycolor1}{rgb}{1.00000,0.00000,1.00000}%
\definecolor{mycolor2}{rgb}{0.00000,1.00000,1.00000}%
\begin{tikzpicture}

\begin{axis}[%
width=1.7in,
height=1.365625in,
at={(0.758333in,0.48125in)},
scale only axis,
xmin=-0.5,
xmax=0.5,
xlabel={Frequency},
ymin=0,
ymax=1,
ylabel={Probability of detection},
title style={font=\bfseries},
axis x line*=bottom,
axis y line*=left,
legend style={at={(0.082602,0.78622)},anchor=south west,legend cell align=left,align=left,draw=white!15!black}
]
\addplot [color=blue,dash pattern=on 1pt off 3pt on 3pt off 3pt,line width=2.5pt]
  table[row sep=crcr]{%
-0.5	0.7208\\
-0.441176470588235	0.7196\\
-0.382352941176471	0.698\\
-0.323529411764706	0.6708\\
-0.264705882352941	0.627\\
-0.205882352941176	0.5728\\
-0.147058823529412	0.4728\\
-0.0882352941176471	0.3276\\
-0.0294117647058824	0.0842\\
0.0294117647058824	0.0794\\
0.0882352941176471	0.2504\\
0.147058823529412	0.4292\\
0.205882352941177	0.5116\\
0.264705882352941	0.4266\\
0.323529411764706	0.438\\
0.382352941176471	0.6536\\
0.441176470588235	0.7144\\
0.5	0.725\\
};
\addlegendentry{Normalized Burg};

\addplot [color=red,solid,line width=2.5pt,mark=square,mark options={solid}]
  table[row sep=crcr]{%
-0.5	0.7034\\
-0.441176470588235	0.7022\\
-0.382352941176471	0.679\\
-0.323529411764706	0.6586\\
-0.264705882352941	0.6142\\
-0.205882352941176	0.545\\
-0.147058823529412	0.444\\
-0.0882352941176471	0.277\\
-0.0294117647058824	0.0944\\
0.0294117647058824	0.0972\\
0.0882352941176471	0.2478\\
0.147058823529412	0.4026\\
0.205882352941177	0.5066\\
0.264705882352941	0.571\\
0.323529411764706	0.6368\\
0.382352941176471	0.6682\\
0.441176470588235	0.6978\\
0.5	0.704\\
};
\addlegendentry{2-step Euclidean Median};

\addplot [color=mycolor1,solid,line width=2.5pt]
  table[row sep=crcr]{%
-0.5	0.6952\\
-0.441176470588235	0.6892\\
-0.382352941176471	0.6758\\
-0.323529411764706	0.6514\\
-0.264705882352941	0.6056\\
-0.205882352941176	0.5482\\
-0.147058823529412	0.4466\\
-0.0882352941176471	0.298\\
-0.0294117647058824	0.0842\\
0.0294117647058824	0.0794\\
0.0882352941176471	0.2874\\
0.147058823529412	0.4134\\
0.205882352941177	0.501\\
0.264705882352941	0.5646\\
0.323529411764706	0.6254\\
0.382352941176471	0.6554\\
0.441176470588235	0.686\\
0.5	0.6958\\
};
\addlegendentry{2-step Poincare Median};

\addplot [color=mycolor2,solid,line width=2.5pt,mark=triangle,mark options={solid,rotate=180}]
  table[row sep=crcr]{%
-0.5	0.7192\\
-0.441176470588235	0.717\\
-0.382352941176471	0.6974\\
-0.323529411764706	0.6656\\
-0.264705882352941	0.6248\\
-0.205882352941176	0.5638\\
-0.147058823529412	0.4548\\
-0.0882352941176471	0.2872\\
-0.0294117647058824	0.0842\\
0.0294117647058824	0.0794\\
0.0882352941176471	0.2846\\
0.147058823529412	0.4598\\
0.205882352941177	0.527\\
0.264705882352941	0.4282\\
0.323529411764706	0.4154\\
0.382352941176471	0.6358\\
0.441176470588235	0.7014\\
0.5	0.718\\
};
\addlegendentry{FP};

\addplot [color=black,solid,line width=2.5pt]
  table[row sep=crcr]{%
-0.5	0.739\\
-0.441176470588235	0.733\\
-0.382352941176471	0.7108\\
-0.323529411764706	0.6834\\
-0.264705882352941	0.632\\
-0.205882352941176	0.564\\
-0.147058823529412	0.4496\\
-0.0882352941176471	0.2698\\
-0.0294117647058824	0.0842\\
0.0294117647058824	0.0794\\
0.0882352941176471	0.259\\
0.147058823529412	0.4362\\
0.205882352941177	0.5632\\
0.264705882352941	0.6386\\
0.323529411764706	0.6972\\
0.382352941176471	0.7096\\
0.441176470588235	0.734\\
0.5	0.7386\\
};
\addlegendentry{Ideal};

\addplot [color=green,dotted,line width=2.5pt]
  table[row sep=crcr]{%
-0.5	0.3512\\
-0.441176470588235	0.362\\
-0.382352941176471	0.341\\
-0.323529411764706	0.309\\
-0.264705882352941	0.2256\\
-0.205882352941176	0.1508\\
-0.147058823529412	0.0564\\
-0.0882352941176471	0.008\\
-0.0294117647058824	0.0022\\
0.0294117647058824	0.0026\\
0.0882352941176471	0.0056\\
0.147058823529412	0.0248\\
0.205882352941177	0.0316\\
0.264705882352941	0.0618\\
0.323529411764706	0.1026\\
0.382352941176471	0.1614\\
0.441176470588235	0.2684\\
0.5	0.3464\\
};
\addlegendentry{OS-CFAR};

\end{axis}
\end{tikzpicture}%
}
\subfigure[$10$ contaminating range cells with normalized frequency $0.3$ ($\mu_1=0.9$, SCR=$0$dB)]{\scriptsize
%
%
\definecolor{mycolor1}{rgb}{1.00000,0.00000,1.00000}%
\definecolor{mycolor2}{rgb}{0.00000,1.00000,1.00000}%
\begin{tikzpicture}

\begin{axis}[%
width=1.7in,
height=1.365625in,
at={(0.758333in,0.48125in)},
scale only axis,
xmin=-0.5,
xmax=0.5,
xlabel={Frequency},
ymin=0,
ymax=1,
ylabel={Probability of detection},
title style={font=\bfseries},
axis x line*=bottom,
axis y line*=left,
legend style={at={(0.082602,0.78622)},anchor=south west,legend cell align=left,align=left,draw=white!15!black}
]
\addplot [color=blue,dash pattern=on 1pt off 3pt on 3pt off 3pt,line width=2.5pt]
  table[row sep=crcr]{%
-0.5	0.744\\
-0.441176470588235	0.7384\\
-0.382352941176471	0.726\\
-0.323529411764706	0.7124\\
-0.264705882352941	0.6692\\
-0.205882352941176	0.6102\\
-0.147058823529412	0.52\\
-0.0882352941176471	0.3118\\
-0.0294117647058824	0.0518\\
0.0294117647058824	0.0102\\
0.0882352941176471	0.1748\\
0.147058823529412	0.4212\\
0.205882352941177	0.5312\\
0.264705882352941	0.4\\
0.323529411764706	0.4594\\
0.382352941176471	0.6782\\
0.441176470588235	0.7494\\
0.5	0.7442\\
};
\addlegendentry{Normalized Burg};

\addplot [color=red,solid,line width=2.5pt,mark=square,mark options={solid}]
  table[row sep=crcr]{%
-0.5	0.7154\\
-0.441176470588235	0.7156\\
-0.382352941176471	0.708\\
-0.323529411764706	0.6912\\
-0.264705882352941	0.6376\\
-0.205882352941176	0.563\\
-0.147058823529412	0.465\\
-0.0882352941176471	0.2302\\
-0.0294117647058824	0.019\\
0.0294117647058824	0.0148\\
0.0882352941176471	0.2002\\
0.147058823529412	0.3906\\
0.205882352941177	0.533\\
0.264705882352941	0.6178\\
0.323529411764706	0.6682\\
0.382352941176471	0.6848\\
0.441176470588235	0.7222\\
0.5	0.7152\\
};
\addlegendentry{2-step Euclidean Median};

\addplot [color=mycolor1,solid,line width=2.5pt]
  table[row sep=crcr]{%
-0.5	0.6916\\
-0.441176470588235	0.6988\\
-0.382352941176471	0.6912\\
-0.323529411764706	0.6728\\
-0.264705882352941	0.6178\\
-0.205882352941176	0.5468\\
-0.147058823529412	0.4464\\
-0.0882352941176471	0.2214\\
-0.0294117647058824	0.023\\
0.0294117647058824	0.0278\\
0.0882352941176471	0.2112\\
0.147058823529412	0.3786\\
0.205882352941177	0.507\\
0.264705882352941	0.588\\
0.323529411764706	0.6398\\
0.382352941176471	0.662\\
0.441176470588235	0.694\\
0.5	0.6944\\
};
\addlegendentry{2-step Poincare Median};

\addplot [color=mycolor2,solid,line width=2.5pt,mark=triangle,mark options={solid,rotate=180}]
  table[row sep=crcr]{%
-0.5	0.7248\\
-0.441176470588235	0.7264\\
-0.382352941176471	0.7226\\
-0.323529411764706	0.7026\\
-0.264705882352941	0.6534\\
-0.205882352941176	0.5924\\
-0.147058823529412	0.4748\\
-0.0882352941176471	0.2408\\
-0.0294117647058824	0.0186\\
0.0294117647058824	0.0168\\
0.0882352941176471	0.2732\\
0.147058823529412	0.487\\
0.205882352941177	0.5646\\
0.264705882352941	0.419\\
0.323529411764706	0.4186\\
0.382352941176471	0.6466\\
0.441176470588235	0.7206\\
0.5	0.7222\\
};
\addlegendentry{FP};

\addplot [color=black,solid,line width=2.5pt]
  table[row sep=crcr]{%
-0.5	0.7334\\
-0.441176470588235	0.742\\
-0.382352941176471	0.7342\\
-0.323529411764706	0.7098\\
-0.264705882352941	0.6524\\
-0.205882352941176	0.5704\\
-0.147058823529412	0.4466\\
-0.0882352941176471	0.199\\
-0.0294117647058824	0.007\\
0.0294117647058824	0.0072\\
0.0882352941176471	0.2066\\
0.147058823529412	0.4316\\
0.205882352941177	0.581\\
0.264705882352941	0.6638\\
0.323529411764706	0.7088\\
0.382352941176471	0.7204\\
0.441176470588235	0.7462\\
0.5	0.7324\\
};
\addlegendentry{Ideal};

\addplot [color=green,dotted,line width=2.5pt]
  table[row sep=crcr]{%
-0.5	0.6628\\
-0.441176470588235	0.668\\
-0.382352941176471	0.6752\\
-0.323529411764706	0.6362\\
-0.264705882352941	0.5526\\
-0.205882352941176	0.4346\\
-0.147058823529412	0.2648\\
-0.0882352941176471	0.0404\\
-0.0294117647058824	0.0018\\
0.0294117647058824	0.002\\
0.0882352941176471	0.0246\\
0.147058823529412	0.1294\\
0.205882352941177	0.2274\\
0.264705882352941	0.3396\\
0.323529411764706	0.442\\
0.382352941176471	0.528\\
0.441176470588235	0.6164\\
0.5	0.6654\\
};
\addlegendentry{OS-CFAR};

\end{axis}
\end{tikzpicture}%
}
\subfigure[Detection probability vs. Signal Clutter Ratio (SCR) with a target of frequency $0.3$ ($\mu_1=0.7$, no outliers)]{\label{detection07d}\scriptsize
%
%
\definecolor{mycolor1}{rgb}{1.00000,0.00000,1.00000}%
\definecolor{mycolor2}{rgb}{0.00000,1.00000,1.00000}%
\begin{tikzpicture}

\begin{axis}[%
width=1.7in,
height=1.365625in,
at={(0.758333in,0.48125in)},
scale only axis,
xmin=-30,
xmax=20,
xlabel={SCR (dB)},
ymin=0,
ymax=1,
ylabel={Probability of detection},
title style={font=\bfseries},
title={GLRT detector},
axis x line*=bottom,
axis y line*=left,
legend style={at={(0.082602,0.98622)},anchor=south west,legend cell align=left,align=left,draw=white!15!black}
]
\addplot [color=blue,dash pattern=on 1pt off 3pt on 3pt off 3pt,line width=2.5pt]
  table[row sep=crcr]{%
20	1\\
17.0588235294118	1\\
14.1176470588235	1\\
11.1764705882353	1\\
8.23529411764706	0.9996\\
5.29411764705883	0.9914\\
2.35294117647059	0.962\\
-0.588235294117652	0.8968\\
-3.52941176470588	0.7928\\
-6.47058823529412	0.6606\\
-9.41176470588235	0.5198\\
-12.3529411764706	0.3952\\
-15.2941176470588	0.2842\\
-18.2352941176471	0.2032\\
-21.1764705882353	0.1416\\
-24.1176470588235	0.095\\
-27.0588235294118	0.0626\\
-30	0.0404\\
};
\addlegendentry{Normalized Burg};

\addplot [color=red,solid,line width=2.5pt,mark=square,mark options={solid}]
  table[row sep=crcr]{%
20	1\\
17.0588235294118	1\\
14.1176470588235	1\\
11.1764705882353	1\\
8.23529411764706	0.9988\\
5.29411764705883	0.9866\\
2.35294117647059	0.9506\\
-0.588235294117652	0.8788\\
-3.52941176470588	0.7658\\
-6.47058823529412	0.6312\\
-9.41176470588235	0.4846\\
-12.3529411764706	0.366\\
-15.2941176470588	0.2674\\
-18.2352941176471	0.1882\\
-21.1764705882353	0.1256\\
-24.1176470588235	0.0884\\
-27.0588235294118	0.0598\\
-30	0.037\\
};
\addlegendentry{2-step Euclidean Median};

\addplot [color=mycolor1,solid,line width=2.5pt]
  table[row sep=crcr]{%
20	1\\
17.0588235294118	1\\
14.1176470588235	1\\
11.1764705882353	0.9998\\
8.23529411764706	0.9964\\
5.29411764705883	0.9804\\
2.35294117647059	0.936\\
-0.588235294117652	0.8542\\
-3.52941176470588	0.7256\\
-6.47058823529412	0.5882\\
-9.41176470588235	0.4476\\
-12.3529411764706	0.3322\\
-15.2941176470588	0.2354\\
-18.2352941176471	0.1662\\
-21.1764705882353	0.1118\\
-24.1176470588235	0.0766\\
-27.0588235294118	0.0524\\
-30	0.0332\\
};
\addlegendentry{2-step Poincare Median};

\addplot [color=mycolor2,solid,line width=2.5pt,mark=triangle,mark options={solid,rotate=180}]
  table[row sep=crcr]{%
20	1\\
17.0588235294118	1\\
14.1176470588235	1\\
11.1764705882353	0.9998\\
8.23529411764706	0.9982\\
5.29411764705883	0.9832\\
2.35294117647059	0.9486\\
-0.588235294117652	0.8714\\
-3.52941176470588	0.7602\\
-6.47058823529412	0.6226\\
-9.41176470588235	0.4778\\
-12.3529411764706	0.359\\
-15.2941176470588	0.2594\\
-18.2352941176471	0.1828\\
-21.1764705882353	0.127\\
-24.1176470588235	0.0842\\
-27.0588235294118	0.0574\\
-30	0.0368\\
};
\addlegendentry{FP};

\addplot [color=black,solid,line width=2.5pt]
  table[row sep=crcr]{%
20	1\\
17.0588235294118	1\\
14.1176470588235	1\\
11.1764705882353	1\\
8.23529411764706	0.9998\\
5.29411764705883	0.993\\
2.35294117647059	0.9678\\
-0.588235294117652	0.9072\\
-3.52941176470588	0.8062\\
-6.47058823529412	0.6786\\
-9.41176470588235	0.534\\
-12.3529411764706	0.407\\
-15.2941176470588	0.2942\\
-18.2352941176471	0.2102\\
-21.1764705882353	0.149\\
-24.1176470588235	0.0988\\
-27.0588235294118	0.0672\\
-30	0.045\\
};
\addlegendentry{Ideal};

\addplot [color=green,dotted,line width=2.5pt]
  table[row sep=crcr]{%
20	1\\
17.0588235294118	1\\
14.1176470588235	1\\
11.1764705882353	1\\
8.23529411764706	0.9992\\
5.29411764705883	0.9134\\
2.35294117647059	0.2738\\
-0.588235294117652	0.0182\\
-3.52941176470588	0.0032\\
-6.47058823529412	0.0014\\
-9.41176470588235	0.001\\
-12.3529411764706	0.001\\
-15.2941176470588	0.001\\
-18.2352941176471	0.001\\
-21.1764705882353	0.001\\
-24.1176470588235	0.001\\
-27.0588235294118	0.001\\
-30	0.001\\
};
\addlegendentry{OS-CFAR};

\end{axis}
\end{tikzpicture}%
}
\caption{Detection probability of a target vs. normalized frequency of the target or Signal-Clutter Ratio (SCR) with a clutter-to-noise ratio CNR=$40$dB for a fixed PFA = $10^{-3}$ (GLRT detector).}
\label{detection07}
\end{figure*}

Figures \ref{detection07}-\ref{detection07bis} take into account a scenario where $N-N_{out}$ neighbor cells are simulated through an AR(1) of parameter $\mu_1$ (considered as clutter) and $N_{out}$ contaminating cells are simulated through AR(1) of parameter $\mu_1 e^{0.3\times 2i\pi}$ of same power (considered as contaminating cells). Moreover, we insert a target at different frequencies represented in the $x$-axis in the cell under test.\\
Figures \ref{detection07}-\ref{detection07bis} illustrate the Doppler resolution of the detectors with respect to the clutter and the contaminating cells. A high frequency resolution is crucial in order to be able to detect small targets (with moderate power) with velocity similar to the (ambient or contaminating) clutter's.
We observe in Figures \ref{detection07}-\ref{detection07bis} that, since the Signal-Clutter Ratio is $0dB$, the probability of detection for frequencies close to the peak frequency of the clutter (i.e. 0) is close to $0$. For normalized frequencies close to $0.3$, only the performances of FP and Normalized Burg estimators as well as OS-CFAR detector decrease which illustrate the robustness of 2-step and Median procedures with respect to outliers. With $10$ contaminating range cells, the gap of detection probability can go up to $30\%$.\\
We illustrate in Figure \ref{detection07d} and \ref{figard}the detection probability with respect to the Signal-Clutter Ratio for a fixed target velocity. These two figures shows only slight differences between studied estimators (at the exception of OS-CFAR) in the non-contaminated case since the estimators are all independent with respect to the non-Gaussian texture.\\
Moreover, Burg estimators have a better frequency resolution with respect to the classical OS-CFAR estimation for the same reason explained in Section \ref{section_est}.\\



\subsection{Comparison to a new geometrical detector}

Instead of using the GLRT statistics, we consider the geometrical detector as alternative for Normalized Burg and 2-step Burg estimators :
\begin{equation}
\text{AR}(\z) = \sum_{k=1}^M (M-k+1) d_P(\hat\mu_k(\z), \hat\mu_{k,amb})^2
\end {equation}
where $(\hat\mu_{1,amb}, ..., \hat\mu_{M,amb})$ are the estimated reflection coefficients of the $N$ surrounding cells and ($\hat\mu_1(\z)$,..., $\hat\mu_M(\z)$) is a regularized estimation of the underlying autoregressive process of the cell under test $\z\in\mathbb{C}^d$ (see for example \cite{barbaresco96}). Let us recall that $d_P$ is the Riemannian distance in the Poincar\'e disc considered in Section \ref{section_poincaremean}.\\
This detector does not directly provide an estimation of the normalized frequency of the detected target. However, its performances, illustrated by Figure \ref{detection07bis}, show competitive results with respect to the same scenario for the GLRT detector with a reduced complexity (neither an inversion of a $d\times d$ matrix nor a $\max$ computation are needed for this detector).

\begin{figure*}[!t]
\centering
\subfigure[No outliers ($\mu_1=0.7$, SCR=$0$dB)]{\scriptsize
%
%
\definecolor{mycolor1}{rgb}{1.00000,0.00000,1.00000}%
\begin{tikzpicture}

\begin{axis}[%
width=1.7in,
height=1.365625in,
at={(0.758333in,0.48125in)},
scale only axis,
xmin=-0.5,
xmax=0.5,
xlabel={Frequency},
ymin=0,
ymax=1,
ylabel={Probability of detection},
title style={font=\bfseries},
axis x line*=bottom,
axis y line*=left,
legend style={at={(0.062602,0.92622)},anchor=south west,legend cell align=left,align=left,draw=white!15!black}
]
\addplot [color=red,solid,line width=2.5pt,mark=diamond,mark options={solid}]
  table[row sep=crcr]{%
-0.5	0.7952\\
-0.441176470588235	0.7866\\
-0.382352941176471	0.7656\\
-0.323529411764706	0.7378\\
-0.264705882352941	0.6778\\
-0.205882352941176	0.5848\\
-0.147058823529412	0.4386\\
-0.0882352941176471	0.2102\\
-0.0294117647058824	0.0344\\
0.0294117647058824	0.04\\
0.0882352941176471	0.215\\
0.147058823529412	0.4414\\
0.205882352941177	0.5916\\
0.264705882352941	0.681\\
0.323529411764706	0.7332\\
0.382352941176471	0.7654\\
0.441176470588235	0.791\\
0.5	0.788\\
};
\addlegendentry{2-step Euclidean Median};

\addplot [color=mycolor1,solid,line width=2.5pt]
  table[row sep=crcr]{%
-0.5	0.7706\\
-0.441176470588235	0.761\\
-0.382352941176471	0.7406\\
-0.323529411764706	0.7066\\
-0.264705882352941	0.6508\\
-0.205882352941176	0.5514\\
-0.147058823529412	0.4096\\
-0.0882352941176471	0.2026\\
-0.0294117647058824	0.05\\
0.0294117647058824	0.0592\\
0.0882352941176471	0.2098\\
0.147058823529412	0.408\\
0.205882352941177	0.5594\\
0.264705882352941	0.652\\
0.323529411764706	0.7076\\
0.382352941176471	0.7398\\
0.441176470588235	0.7684\\
0.5	0.7648\\
};
\addlegendentry{2-step Poincare Median};

\addplot [color=blue,dash pattern=on 1pt off 3pt on 3pt off 3pt,line width=2.5pt]
  table[row sep=crcr]{%
-0.5	0.7922\\
-0.441176470588235	0.7818\\
-0.382352941176471	0.762\\
-0.323529411764706	0.7346\\
-0.264705882352941	0.672\\
-0.205882352941176	0.5796\\
-0.147058823529412	0.4336\\
-0.0882352941176471	0.2028\\
-0.0294117647058824	0.0322\\
0.0294117647058824	0.0368\\
0.0882352941176471	0.2056\\
0.147058823529412	0.4312\\
0.205882352941177	0.5816\\
0.264705882352941	0.6722\\
0.323529411764706	0.7288\\
0.382352941176471	0.7638\\
0.441176470588235	0.785\\
0.5	0.7824\\
};
\addlegendentry{Normalized Burg};

\addplot [color=green,dotted,line width=2.5pt]
  table[row sep=crcr]{%
-0.5	0.4384\\
-0.441176470588235	0.4306\\
-0.382352941176471	0.3756\\
-0.323529411764706	0.3172\\
-0.264705882352941	0.2150\\
-0.205882352941176	0.1226\\
-0.147058823529412	0.0422\\
-0.0882352941176471	0.0052\\
-0.0294117647058824	0.0022\\
0.0294117647058824	0.0032\\
0.0882352941176471	0.0068\\
0.147058823529412	0.0446\\
0.205882352941177	0.1198\\
0.264705882352941	0.2156\\
0.323529411764706	0.3202\\
0.382352941176471	0.3670\\
0.441176470588235	0.4200\\
0.5	0.4400\\
};
\addlegendentry{OS-CFAR};
    
\end{axis}
\end{tikzpicture}%
}
\subfigure[$10$ contaminating range cells with normalized frequency $0.3$ ($\mu_1=0.7$, SCR=$0$dB)]{\scriptsize
%
%
\definecolor{mycolor1}{rgb}{1.00000,0.00000,1.00000}%
\begin{tikzpicture}

\begin{axis}[%
width=1.7in,
height=1.365625in,
at={(0.758333in,0.48125in)},
scale only axis,
xmin=-0.5,
xmax=0.5,
xlabel={Frequency},
ymin=0,
ymax=1,
ylabel={Probability of detection},
title style={font=\bfseries},
axis x line*=bottom,
axis y line*=left,
legend style={at={(0.062602,0.92622)},anchor=south west,legend cell align=left,align=left,draw=white!15!black}
]
\addplot [color=red,solid,line width=2.5pt,mark=diamond,mark options={solid}]
  table[row sep=crcr]{%
-0.5	0.8102\\
-0.441176470588235	0.815\\
-0.382352941176471	0.799\\
-0.323529411764706	0.7804\\
-0.264705882352941	0.7282\\
-0.205882352941176	0.6488\\
-0.147058823529412	0.5054\\
-0.0882352941176471	0.2738\\
-0.0294117647058824	0.07\\
0.0294117647058824	0.0552\\
0.0882352941176471	0.2392\\
0.147058823529412	0.4524\\
0.205882352941177	0.615\\
0.264705882352941	0.6924\\
0.323529411764706	0.748\\
0.382352941176471	0.7742\\
0.441176470588235	0.81\\
0.5	0.82\\
};
\addlegendentry{2-step Euclidean Median};

\addplot [color=mycolor1,solid,line width=2.5pt]
  table[row sep=crcr]{%
-0.5	0.7734\\
-0.441176470588235	0.77\\
-0.382352941176471	0.7538\\
-0.323529411764706	0.7246\\
-0.264705882352941	0.6712\\
-0.205882352941176	0.5766\\
-0.147058823529412	0.4302\\
-0.0882352941176471	0.2234\\
-0.0294117647058824	0.0668\\
0.0294117647058824	0.0668\\
0.0882352941176471	0.2242\\
0.147058823529412	0.4144\\
0.205882352941177	0.566\\
0.264705882352941	0.6506\\
0.323529411764706	0.7066\\
0.382352941176471	0.736\\
0.441176470588235	0.7722\\
0.5	0.7828\\
};
\addlegendentry{2-step Poincare Median};

\addplot [color=blue,dash pattern=on 1pt off 3pt on 3pt off 3pt,line width=2.5pt]
  table[row sep=crcr]{%
-0.5	0.7034\\
-0.441176470588235	0.7256\\
-0.382352941176471	0.7252\\
-0.323529411764706	0.7124\\
-0.264705882352941	0.6654\\
-0.205882352941176	0.5778\\
-0.147058823529412	0.447\\
-0.0882352941176471	0.2678\\
-0.0294117647058824	0.0982\\
0.0294117647058824	0.0386\\
0.0882352941176471	0.0936\\
0.147058823529412	0.2458\\
0.205882352941177	0.384\\
0.264705882352941	0.4762\\
0.323529411764706	0.5358\\
0.382352941176471	0.5978\\
0.441176470588235	0.661\\
0.5	0.7092\\
};
\addlegendentry{Normalized Burg};

\addplot [color=green,dotted,line width=2.5pt]
  table[row sep=crcr]{%
-0.5	0.3532\\
-0.441176470588235	0.3560\\
-0.382352941176471	0.3598\\
-0.323529411764706	0.3070\\
-0.264705882352941	0.2392\\
-0.205882352941176	0.1454\\
-0.147058823529412	0.0556\\
-0.0882352941176471	0.0096\\
-0.0294117647058824	0.0026\\
0.0294117647058824	0.0020\\
0.0882352941176471	0.0074\\
0.147058823529412	0.0258\\
0.205882352941177	0.0410\\
0.264705882352941	0.0682\\
0.323529411764706	0.1088\\
0.382352941176471	0.1812\\
0.441176470588235	0.2700\\
0.5	0.3562\\
};
\addlegendentry{OS-CFAR};

\end{axis}
\end{tikzpicture}%
}
\subfigure[$10$ contaminating range cells with normalized frequency $0.3$ ($\mu_1=0.9$, SCR=$0$dB)]{\scriptsize
%
%
\definecolor{mycolor1}{rgb}{1.00000,0.00000,1.00000}%
\begin{tikzpicture}

\begin{axis}[%
width=1.7in,
height=1.365625in,
at={(0.758333in,0.48125in)},
scale only axis,
xmin=-0.5,
xmax=0.5,
xlabel={Frequency},
ymin=0,
ymax=1,
ylabel={Probability of detection},
title style={font=\bfseries},
axis x line*=bottom,
axis y line*=left,
legend style={at={(0.062602,0.92622)},anchor=south west,legend cell align=left,align=left,draw=white!15!black}
]
\addplot [color=red,solid,line width=2.5pt,mark=diamond,mark options={solid}]
  table[row sep=crcr]{%
-0.5	0.9262\\
-0.441176470588235	0.928\\
-0.382352941176471	0.9238\\
-0.323529411764706	0.9076\\
-0.264705882352941	0.8828\\
-0.205882352941176	0.8346\\
-0.147058823529412	0.7468\\
-0.0882352941176471	0.496\\
-0.0294117647058824	0.06\\
0.0294117647058824	0.0402\\
0.0882352941176471	0.4658\\
0.147058823529412	0.7158\\
0.205882352941177	0.81\\
0.264705882352941	0.8628\\
0.323529411764706	0.8976\\
0.382352941176471	0.912\\
0.441176470588235	0.9212\\
0.5	0.9294\\
};
\addlegendentry{2-step Euclidean Median};

\addplot [color=mycolor1,solid,line width=2.5pt]
  table[row sep=crcr]{%
-0.5	0.8638\\
-0.441176470588235	0.8654\\
-0.382352941176471	0.8668\\
-0.323529411764706	0.8436\\
-0.264705882352941	0.7992\\
-0.205882352941176	0.7246\\
-0.147058823529412	0.5838\\
-0.0882352941176471	0.3034\\
-0.0294117647058824	0.0264\\
0.0294117647058824	0.0198\\
0.0882352941176471	0.2874\\
0.147058823529412	0.561\\
0.205882352941177	0.6872\\
0.264705882352941	0.7742\\
0.323529411764706	0.8204\\
0.382352941176471	0.838\\
0.441176470588235	0.8616\\
0.5	0.8716\\
};
\addlegendentry{2-step Poincare Median};

\addplot [color=blue,dash pattern=on 1pt off 3pt on 3pt off 3pt,line width=2.5pt]
  table[row sep=crcr]{%
-0.5	0.7812\\
-0.441176470588235	0.81\\
-0.382352941176471	0.829\\
-0.323529411764706	0.8142\\
-0.264705882352941	0.7672\\
-0.205882352941176	0.698\\
-0.147058823529412	0.555\\
-0.0882352941176471	0.3032\\
-0.0294117647058824	0.0422\\
0.0294117647058824	0.0004\\
0.0882352941176471	0.0804\\
0.147058823529412	0.326\\
0.205882352941177	0.4978\\
0.264705882352941	0.5974\\
0.323529411764706	0.6632\\
0.382352941176471	0.6912\\
0.441176470588235	0.7588\\
0.5	0.793\\
};
\addlegendentry{Normalized Burg};

\addplot [color=green,dotted,line width=2.5pt]
  table[row sep=crcr]{%
-0.5	0.6628\\
-0.441176470588235	0.668\\
-0.382352941176471	0.6752\\
-0.323529411764706	0.6362\\
-0.264705882352941	0.5526\\
-0.205882352941176	0.4346\\
-0.147058823529412	0.2648\\
-0.0882352941176471	0.0404\\
-0.0294117647058824	0.0018\\
0.0294117647058824	0.002\\
0.0882352941176471	0.0246\\
0.147058823529412	0.1294\\
0.205882352941177	0.2274\\
0.264705882352941	0.3396\\
0.323529411764706	0.442\\
0.382352941176471	0.528\\
0.441176470588235	0.6164\\
0.5	0.6654\\
};
\addlegendentry{OS-CFAR};

\end{axis}
\end{tikzpicture}%
}
\subfigure[Detection probability vs. Signal Clutter Ratio (SCR) with a target of frequency $0.3$ ($\mu_1=0.7$, no outliers)]{\label{figard}\scriptsize
%
%
\definecolor{mycolor1}{rgb}{1.00000,0.00000,1.00000}%
\begin{tikzpicture}

\begin{axis}[%
width=1.7in,
height=1.365625in,
at={(0.758333in,0.48125in)},
scale only axis,
xmin=-30,
xmax=20,
xlabel={SCR (dB)},
ymin=0,
ymax=1,
ylabel={Probability of detection},
title style={font=\bfseries},
title={AR detector},
axis x line*=bottom,
axis y line*=left,
legend style={at={(0.048549,0.999293)},anchor=south west,legend cell align=left,align=left,draw=white!15!black}
]
\addplot [color=mycolor1,solid,line width=2.5pt,mark=diamond,mark options={solid}]
  table[row sep=crcr]{%
20	1\\
17.0588235294118	1\\
14.1176470588235	1\\
11.1764705882353	1\\
8.23529411764706	0.9964\\
5.29411764705883	0.9752\\
2.35294117647059	0.9228\\
-0.588235294117652	0.8134\\
-3.52941176470588	0.6552\\
-6.47058823529412	0.4756\\
-9.41176470588235	0.3362\\
-12.3529411764706	0.2094\\
-15.2941176470588	0.1318\\
-18.2352941176471	0.0772\\
-21.1764705882353	0.0432\\
-24.1176470588235	0.024\\
-27.0588235294118	0.0122\\
-30	0.0036\\
};
\addlegendentry{2-step Euclidean Median};

\addplot [color=red,solid,line width=2.5pt]
  table[row sep=crcr]{%
20	1\\
17.0588235294118	1\\
14.1176470588235	1\\
11.1764705882353	0.9994\\
8.23529411764706	0.9938\\
5.29411764705883	0.9672\\
2.35294117647059	0.9046\\
-0.588235294117652	0.7774\\
-3.52941176470588	0.6128\\
-6.47058823529412	0.4402\\
-9.41176470588235	0.295\\
-12.3529411764706	0.185\\
-15.2941176470588	0.1158\\
-18.2352941176471	0.0648\\
-21.1764705882353	0.0358\\
-24.1176470588235	0.0194\\
-27.0588235294118	0.0084\\
-30	0.0016\\
};
\addlegendentry{2-step Poincare Median};

\addplot [color=blue,dash pattern=on 1pt off 3pt on 3pt off 3pt,line width=2.5pt]
  table[row sep=crcr]{%
20	1\\
17.0588235294118	1\\
14.1176470588235	1\\
11.1764705882353	1\\
8.23529411764706	0.9972\\
5.29411764705883	0.9772\\
2.35294117647059	0.9268\\
-0.588235294117652	0.822\\
-3.52941176470588	0.6612\\
-6.47058823529412	0.4906\\
-9.41176470588235	0.341\\
-12.3529411764706	0.2164\\
-15.2941176470588	0.1382\\
-18.2352941176471	0.0778\\
-21.1764705882353	0.0428\\
-24.1176470588235	0.0246\\
-27.0588235294118	0.0126\\
-30	0.0044\\
};
\addlegendentry{Normalized Burg};

\addplot [color=green,dotted,line width=2.5pt]
  table[row sep=crcr]{%
20	1\\
17.0588235294118	1\\
14.1176470588235	1\\
11.1764705882353	1\\
8.23529411764706	0.9992\\
5.29411764705883	0.9134\\
2.35294117647059	0.2738\\
-0.588235294117652	0.0182\\
-3.52941176470588	0.0032\\
-6.47058823529412	0.0014\\
-9.41176470588235	0.001\\
-12.3529411764706	0.001\\
-15.2941176470588	0.001\\
-18.2352941176471	0.001\\
-21.1764705882353	0.001\\
-24.1176470588235	0.001\\
-27.0588235294118	0.001\\
-30	0.001\\
};
\addlegendentry{OS-CFAR};

\end{axis}
\end{tikzpicture}%
}
\caption{Detection probability of a target vs. normalized frequency of the target or Signal-Clutter Ration (SCR) with a clutter-to-noise ratio CNR=$40$dB for a fixed PFA = $10^{-3}$ (AR detector).}
\label{detection07bis}
\end{figure*}

\subsection{CFAR property of detectors}

We illustrate the CFAR property of different detectors in Table \ref{table_cfar1} by computing the true probability of false alarm (PFA) of the test statistics for a unique threshold. This property is crucial for the detectors to be stable on real operations.\\
Since all proposed estimates are independent with respect to texture realizations (neglecting thermal noise), the CFAR property with respect to the texture distribution holds for both GLRT and AR detectors.\\
Moreover, the GLRT detector coupled with Fixed Point estimator was shown to have CFAR property with respect to the clutter covariance (see e.g. \cite{conte04b} which also contains validation of this property on real radar measurements); it is confirmed by Table \ref{table_cfar1} where only thermal noise have a small impact on the PFA in the non-contaminated case. The presence of contaminating range cells, however, breaks the CFAR property, especially for non-robust estimators (Normalized Burg and FP).\\
Table shows that GLRT detector is more robust than AR detector with respect to CFAR property.\\
Let us note furthermore that the GLRT detector is more robust than AR detector with respect to CFAR property. This could lead to further improvements of the latter.

\begin{table}[!t]
\center
\caption{False alarm probability (to be multiplied by $10^{-3}$) for a fixed threshold computed through Monte-Carlo simulations for the scenario with $\mu_1=0.7$ without contamination; x/y means PFA=$x.10^{-3}$ for the GLRT detector and PFA = $y.10^{-3}$ for the AR detector (unavailable for FP estimator).} \begin{tabular}{|c|c|c|c|c|c|}
   \hline
 &\parbox{2.5cm}{\centerline{Normalized Burg}}   & \parbox{2.9cm}{{2-step Euclidean\\  Median Burg}} & \parbox{2.7cm}{2-step Poincar\'e\\ Median Burg} & \parbox{2.3cm}{\centerline{FP}} \\
\hline
\parbox{4cm}{no contamination\\ ($\mu_1=0.3$)} & $0.8/1.2$ & $2.5/1.2$ & $0.6/1.1$ & $0.6/$NA \\
\parbox{4cm}{no contamination\\ ($\mu_1=0.7$)} & $1.0/1.0$ & $1.0/1.0$ & $1.0/1.0$ & $1.0/$NA \\
\parbox{4cm}{no contamination\\ ($\mu_1=0.9$)} &  $2.8/8.8$ & $3.0/12.0$ & $2.6/5.0$ & $1.6/$NA \\
\parbox{4cm}{$20$ contaminating range\\ cells ($\mu_1=0.7$)} & $5.8/15.2$ & $3.2/0.40$ & $2.2/0.60$ & $3.2/$NA \\
\parbox{4cm}{$20$ contaminating range\\ cells ($\mu_1=0.9$)} & $15.0/164.8$ & $3.8/4.0$ & $4.0/5.6$ & $5.8/$NA \\
\hline
\end{tabular}
\label{table_cfar1}
\end{table}

\section{Conclusion}

We have presented Burg methods for mixtures of autoregressive vectors which are a natural family of distributions when we consider non-Gaussian stationary clutter. We proposed several estimators independent of the non-Gaussian texture in this framework and studied their behavior especially in terms of robustness with respect to contamination and efficiency. We can sum up these through the following insights:

We have presented several Burg methods for mixtures of autoregressive vectors which are a natural family of distributions when we consider non-Gaussian stationary clutter. We proposed several estimators independent of the non-Gaussian texture in this framework and studied their behavior especially in terms of robustness with respect to contamination and efficiency. We can sum up these through the following insights:
\begin{itemize}
\item Thanks to the exploitation of the Toeplitz structure, Burg methods complexity goes \textbf{from} $\mathbf{O(d^3)}$ (that corresponds to the complexity of the inversion of a matrix $d\times d$) \textbf{to} $\mathbf{O(M^2)}$ (recall that $M$ stands for the order of the autoregressive model).
\item It is useful to take into account the non-Gaussianity of the clutter for sub-exponential amplitude distributions.
\item Considering medians instead of means in the Burg estimators furnishes a robustness for medium contamination ($10\%$ to $30\%$ outlier samples).
\item 2-step procedures (consisting in a selection of secondary data) need \textbf{a first estimation of the scatter matrix of the clutter that is robust enough in order to be efficient}. In that case, high contamination (close to $50\%$) can be considered.
\item Taking into account the Poincar\'e metric is efficient for reflection parameters of high order (that should be close to $0$) since they tends to under-estimate their modulus.
\end{itemize}

Future works will be devoted to the extension of these methods for non-stationary signal in the burst \cite{lebrigant14,ruiz15}.
This work will be also extended for STAP (Space-Time Adaptive Processing) based on OS-STAP algorithm described in \cite{barbaresco12} by computing mean and median Toeplitz-Block-Toeplitz covariance matrices parameterized by Matrix-Valued Autoregressive model, and based on numerical scheme described in \cite{jeuris15,jeuris15b,jeuris15c}.


\section*{Acknowledgment}

The authors would like to thank the DGA/MRIS for its support as well as the referee for useful comments.

\end{document}